\documentclass[preprint,12pt]{elsarticle}

\usepackage{amssymb}
\usepackage{amsmath}

\usepackage{graphicx}
\usepackage{mathrsfs}
\usepackage{fix-cm}
\usepackage{latexsym}
\usepackage{xcolor}
\usepackage{bbm}
\usepackage{enumitem}
\usepackage{ulem}
\usepackage{amsthm}
\usepackage{xcolor}

\usepackage{tikz}
\usepackage{silence}
\usepackage{caption}

\usepackage{times}

\usepackage{todonotes}

\usepackage{hyperref}

\makeatletter
\def\namedlabel#1#2{\begingroup
	#2%
	\def\@currentlabel{#2}%
	\phantomsection\label{#1}\endgroup
}
\makeatother
\renewcommand{\epsilon}{\varepsilon} 
\renewcommand{\H}{\mathcal{H}}
\newcommand{\R}{\mathbb{R}}

\newcommand{\N}{\mathbb{N}}
\newcommand{\tto}{\rightrightarrows}
\newcommand{\gph}{\operatorname{gph}}

\newcommand{\ov}{\overline}
\def \B{\mathbb{B}}

\newcommand{\cl}{{\operatorname{cl}}}
\newcommand{\inte}{{\operatorname{int \  }}}
\newcommand{\bd}{{\operatorname{bd \ }}}
\newcommand{\var}{{\operatorname{var}}}

\DeclareMathOperator{\proj}{proj}

 \definecolor{OliveGreen}{rgb}{0,0.6,0}

\theoremstyle{plain}
\newtheorem{theorem}{Theorem}[section]
\newtheorem{lemma}[theorem]{Lemma}
\newtheorem{proposition}[theorem]{Proposition}
\newtheorem{corollary}[theorem]{Corollary}

\theoremstyle{definition}
\newtheorem{definition}[theorem]{Definition}

\theoremstyle{remark}
\newtheorem{remark}[theorem]{Remark}

\journal{Nuclear Physics B}

\begin{document}
\begin{frontmatter}



\title{Stochastic Perturbation of Sweeping Processes Driven by Continuous Uniformly Prox-Regular Moving Sets}


\author[a]{Juan Guillermo Garrido}
\ead{jgarrido@dim.uchile.cl}
\author[b]{Nabil Kazi-Tani}
\ead{nabil.kazi-tani@univ-lorraine.fr}
\author[c,d]{Emilio Vilches}
\ead{emilio.vilches@uoh.cl}


\affiliation[a]{organization={Departamento de Ingenier\'ia Matem\'atica,  Universidad de Chile},
            state={Santiago},
            country={Chile}}
\affiliation[b]{organization={Université de Lorraine, CNRS, IECL, F-57000},
            state={Metz},
            country={France}}
\affiliation[c]{organization={Instituto de Ciencias de la Ingenier\'ia, Universidad de O'Higgins },
            state={Rancagua},
            country={Chile}}
            
\affiliation[d]{organization={Centro de Modelamiento Matem\'atico  (CNRS UMI 2807), Universidad de Chile },
            state={Santiago},
            country={Chile}}

\begin{abstract}
In this paper, we study the existence of solutions to sweeping processes in the presence of stochastic perturbations, where the moving set takes uniformly prox-regular values and varies continuously with respect to the Hausdorff distance, without smoothness assumptions. We propose a minimal geometric framework for such moving sets, make precise the logical implications between several standard hypotheses in the literature, and provide practical sufficient conditions that apply in particular to constraints defined as finite intersections of sublevel sets. Within this setting, we establish existence of weak and strong solutions and prove pathwise uniqueness for the associated stochastic differential equations reflected in time-dependent domains.
\end{abstract}




\end{frontmatter}



\section{Introduction}
The sweeping process is a first-order differential inclusion involving the normal cone to a family of moving sets. Introduced by J.J.~Moreau in a series of seminal papers (see, e.g., \cite{MO1,MO2,MR508661}), it was motivated by concrete applications in quasi-static evolution in elastoplasticity, contact mechanics, and friction dynamics (see, e.g., \cite{Moreau_2004,MR513445}). Since then, it has been used to model dynamical systems with time-dependent constraints, notably in nonsmooth mechanics, crowd motion (see, e.g., \cite{Brogliato-M,Maury-Venel}), and electrical-circuit modeling \cite{Acary-Bon-Bro-2011}. In its simplest form, the sweeping process reads as
\begin{equation}\label{sp111111}
  \left\{
  \begin{aligned}
    \dot{x}(t)&\in f(t,x(t)) - N^P(C(t);x(t)) && \text{ for a.e. } t\in[0,T],\\
    x(0)&=x_0\in C(0),
  \end{aligned}
  \right.
\end{equation}
where $\H$ is a Hilbert space, $C\colon[0,T]\tto \H$ is a set-valued map and $f\colon[0,T]\times \H \to \H$ is a mapping. Here $N^P(C(t);x(t))$ denotes the proximal normal cone to the set $C(t)$ at $x(t)\in C(t)$. The above differential inclusion has been thoroughly studied in the case of nonconvex moving sets, notably for uniformly prox-regular sets (a generalization of convexity), for which well-posedness is well established (see, e.g., \cite{MR1994056}).

Parallel to the deterministic sweeping process, reflected stochastic differential equations—and their deterministic counterpart, the Skorokhod problem—constitute a well-established class of objects in probability theory, with a vast literature. They have found numerous applications, notably in constrained stochastic control, PDEs with Neumann boundary conditions, queueing networks, mathematical economics, and game theory.
Building on this line of work, in this paper we study the existence of solutions to the following stochastic sweeping differential inclusion
\begin{equation}\label{spe_intro}
    dX_t\in -N^P(C(t);X_t)dt + \sigma(t,X_t)dB_t + f(t,X_t)dt,
\end{equation}
which corresponds to a stochastically perturbed version of the sweeping process. We work in a finite-dimensional Hilbert space $\H=\mathbb{R}^d$ and consider the case where the moving set $C\colon[0,T]\tto \H$ takes closed, uniformly prox-regular values and depends continuously on time with respect to the Hausdorff distance. To the best of our knowledge, this setting is new and represents the most general framework studied to date for stochastic sweeping processes with uniformly prox-regular moving sets.

The differential inclusion \eqref{spe_intro} can be viewed as a generalization of the Skorokhod problem, which describes the reflection of a given trajectory at the boundary of a moving domain, with the direction of reflection taken to be normal (see, e.g., \cite{Skorokhod1961}). While in the probability theory literature such constrained systems are typically described as stochastic differential equations reflected on time-dependent domains, we adopt for simplicity the terminology \emph{stochastic sweeping process} to emphasize the connection with the deterministic sweeping framework.

For the time-independent case, the stochastic sweeping process reduces to the classical theory of reflected stochastic differential equations and the Skorokhod problem in fixed domains (see, e.g., \cite[Section 1.2]{MR2812587}), a topic that has generated a vast literature, for which we refer the reader to \cite{pilipenko2014introduction} and the references therein. In dimension one, the pathwise reflection construction goes back to Skorokhod \cite{Skorokhod1961}. In several dimensions, Tanaka \cite{MR529332} studied normal reflection in convex domains. Lions and Sznitman \cite{MR745330} then proved existence for reflection in uniformly prox-regular domains. Saisho \cite{MR873889} further extended this line of work by removing a smoothness assumption imposed in \cite{MR745330}. Subsequently, Dupuis and Ishii \cite{dupuis1993sdes} allowed for oblique directions of reflection and non smooth domains.

To the best of the authors' knowledge, the time-dependent case was first addressed in \cite{oshima2004construction}, where the author, using the theory of time-dependent Dirichlet forms, constructed reflecting diffusions on time-dependent domains, given as the image of a fixed domain under a homeomorphic map. In \cite{Costantini2006Boundary}, the authors, motivated by optimal stopping and singular stochastic control, considered a moving set $C$ that varies continuously with respect to the Hausdorff distance, takes uniformly prox-regular values, and satisfies a smoothness condition on the boundary. Subsequently, in \cite{MR2683628}, the Skorokhod problem with oblique reflection in time-dependent domains is studied assuming a continuous reflection direction field and suitable geometric regularity, and existence results for weak solutions of SDEs are given. In contrast, we treat normal reflection on irregular domains, and we additionally prove existence of strong solutions together with pathwise uniqueness under our setting. In \cite{lundstrom2019stochastic}, the authors also impose a continuous oblique reflection direction and Sobolev-in-time regularity of the boundary. They then prove existence and uniqueness of strong solutions to obliquely reflected SDEs and of viscosity solutions to fully nonlinear PDEs with oblique derivative boundary conditions. In \cite{MR2812587}, the results of \cite{MR873889,MR745330} were extended to the time-dependent setting, assuming that the moving set is absolutely continuous with respect to the Hausdorff distance. More recently, in \cite{MR3509666,MR4461030}, existence was established for time-dependent convex moving sets with nonempty interior, without additional geometric conditions, and in the case where the driving noise is rough, which include the case of fractional Brownian motion.

While these contributions have significantly advanced the theory, the case of stochastically perturbed sweeping processes with time-dependent, uniformly prox-regular moving sets has, to the best of our knowledge, not yet been systematically addressed. The purpose of this work is to help fill this gap by establishing existence results within a framework that appears to be the most general studied so far.

The commonly adopted approach to deal with stochastic sweeping processes consists first in studying a deterministic Skorokhod problem (see, e.g., \cite{MR873889, MR529332, MR745330, MR2812587}), which then provides a natural interpretation when the perturbation is stochastic. Our goal is therefore twofold: on the one hand, to study the existence of sweeping process dynamics when the moving set is Hausdorff-continuous and takes uniformly prox-regular values, exploring different geometric hypotheses inspired by the related literature; on the other hand, to extend these existence results to the stochastic framework.


\paragraph{Contributions}
Our contributions are threefold.\\
(i) \textbf{Geometric framework:} We introduce a minimal geometric setup for moving sets with uniformly prox-regular values and merely Hausdorff-continuous time dependence, and we
systematize the assumptions by making their logical implications explicit and
by positioning them within the existing literature. In particular, we identify
what is, to the best of our knowledge, the weakest verifiable condition
available for Skorokhod-type problems, and we provide practical criteria that
apply, in particular, to constraints given as finite intersections of sublevel
sets. This first part furnishes the geometric foundation for the study of
 stochastic sweeping processes.\\
\noindent (ii) \textbf{Deterministic Skorokhod/sweeping process problem:} Within this geometric setting, we
prove new well-posedness results for the deterministic sweeping process
(equivalently, its Skorokhod-type formulation) under continuous perturbations.
We also derive quantitative bounds on the total variation of the correction
term, including uniform estimates for families of equicontinuous perturbations.\\
\noindent (iii) \textbf{Stochastic sweeping process:} We establish existence of weak and
strong solutions, as well as pathwise uniqueness, for stochastic perturbations
of the sweeping process driven by moving constraint sets that are nonconvex and
nonsmooth. To the best of our knowledge, these results hold under the most
general set of assumptions currently available in this framework.

 The paper is organized as follows. In Section \ref{sec2}, we present the necessary mathematical background and notation used throughout the paper. Section \ref{sec3} introduces the main geometric hypotheses considered in the literature, together with examples illustrating their validity. In Section \ref{sec4}, we establish relationships between the geometric hypotheses introduced in Section \ref{sec3}. Section \ref{sec5} recalls an important existence and uniqueness result for the sweeping process from \cite{MR1946545} and provides an estimate of the total variation of the solution for a perturbed version of \eqref{sp111111}. Section \ref{sec7} addresses \eqref{spe_intro} via the Skorokhod problem and establishes a weak existence result under general geometric assumptions, together with pathwise uniqueness in the Lipschitz case, where we also obtain strong existence and uniqueness of solutions to \eqref{spe_intro}. Finally, some technical proofs present in the work are included in the appendix.

\section{Preliminaries}\label{sec2}
	
	Throughout the paper, $\H$ stands for a Hilbert space endowed with the inner product $\langle\cdot,\cdot\rangle$ and norm $\|\cdot\|$. The closed (resp. open) ball centered at $x$ with radius $r>0$ is denoted by $\mathbb{B}_r[x]$ (resp. $\mathbb{B}_r(x)$), and the closed unit ball is denoted by $\mathbb{B}$. For a subset $A\subset\H$, we denote the closure, interior and boundary of $A$, respectively, as $\cl A$, $\inte A$, and $\bd A$. We also set $\mathbb{S} = \{x\in\H : \|x\| = 1\}$.\\
	For a given set $S\subset \H$, the \emph{support function} and the \emph{distance function} of $S$ at $x\in \H$ are defined, respectively, by
	$$
	\sigma(x;S):= \sup_{z\in S}\langle x,z\rangle \quad \textrm{and} \quad d(x;S):=\inf_{z\in S}\|x-z\|.
	$$
    Sometimes, we write $d_S(x):=d(x;S)$ for brevity. The \emph{convex hull}  of $S$ is denoted as $\operatorname{co}(S)$. Given $\rho\in ]0,+\infty]$ and $\gamma\in ]0,1[$, the  \emph{$\rho$-enlargement} and the 
	\emph{$\gamma\rho$-enlargement} of $S$ are 
	$$
	U_\rho(S) := \{x\in \H:d(x;S)<\rho\} \textrm{ and } U_\rho^\gamma(S):=U_{\gamma\rho}(S).
	$$
	Given two sets $A,B\subset \H$, the \emph{excess} of $A$ over $B$ is 
    $$e(A,B) := \sup_{x\in A} d(x;B),$$ and the \emph{Hausdorff distance} between $A$ and $B$ is 
	$$d_H(A,B) := \max\{e(A,B),e(B,A)\}.$$
	Further properties of the Hausdorff distance can be found in \cite[Sec. 3.16]{MR2378491}. The \emph{projection} of $x$ onto $S\subset \H$ is the (possibly empty) set  
$$\operatorname{Proj}_{S}(x):=\left\{z\in S : d_{S}(x)=\Vert x-z\Vert\right\}.$$ When the above set is a singleton, we denote its unique element by $\operatorname{proj}_{S}(x)$.
	
	\noindent Let $f\colon \H\to \mathbb{R}\cup\{+\infty\}$ be a \emph{lower semicontinuous} (lsc) function and let $x\in \operatorname{dom}f$. We say that $\zeta$ belongs to the \emph{proximal subdifferential}   of $f$ at $x$, denoted  $\partial_P f(x)$, if there exist $\sigma\geq 0$ and $\eta\geq 0$ such that
		\begin{equation*}
			f(y)\geq f(x)+\left\langle \zeta,y-x\right\rangle -\sigma\Vert y-x\Vert^2 \textrm{ for all } y\in \mathbb{B}_{\eta}(x).
		\end{equation*}
	Using this notion, the \emph{proximal normal cone} of a set $S\subset \H$ at $x\in S$ is
	$$
	N^P(S;x):=\partial_P I_S(x),
	$$
	where $I_S$ is the indicator function of $S\subset \H$ (i.e., $I_S(x)=0$ if $x\in S$ and $I_S(x)=+\infty$ otherwise). It is well-known (see \cite[Theorem 4.1]{MR1870754}) that, for all $x\in S$, 
	\begin{equation}\label{normal-distancia}
		N^{P}(S;x)\cap \mathbb{B}=\partial_P d_S(x).
	\end{equation}
	
	We recall the concept of uniformly prox-regular sets. It was introduced by Federer in the finite-dimensional case (see \cite{MR110078}) and later developed by Rockafellar, Poliquin, and Thibault \cite{Poliquin2000}.  Prox-regularity generalizes and unifies convexity and nonconvex bodies with $C^2$ boundaries. For surveys, see \cite{MR2768810,Thibault-2023-II}.
	
\begin{definition}
Let $S\subset \H$ be nonempty and closed and $\rho\in ]0,+\infty]$. We say that $S$ is \emph{$\rho$-uniformly prox-regular} if, for all $x\in S$ and $\zeta\in N^P(S;x)$, 
		\begin{equation*}
			\langle \zeta,x'-x\rangle\leq \frac{\|\zeta\|}{2\rho}\|x'-x\|^2 \textrm{ for all } x'\in S.
		\end{equation*}
	\end{definition}
	Note that every closed convex set is $\rho$-uniformly prox-regular for every $\rho>0$. The following proposition provides a characterization of uniformly prox-regular sets 
	(see, e.g.,  \cite{MR2768810}). 
	\begin{proposition}\label{prox_reg_prop}
		Let $S\subset \H$ be a closed set and $\rho\in ]0,+\infty]$. The following assertions are equivalent:
		\begin{enumerate}
			\item [(a)] $S$ is $\rho$-uniformly prox-regular.
			\item [(b)] For any positive $\gamma<1$ the mapping $\proj_S$ is well-defined on $U_\rho^\gamma(S)$ and Lipschitz continuous on $U_\rho^\gamma(S)$ with $(1-\gamma)^{-1}$ as a Lipschitz constant, i.e.,
			$$\left\|\proj_S\left(u_1\right)-\proj_S\left(u_2\right)\right\| \leq(1-\gamma)^{-1}\left\|u_1-u_2\right\| \quad$$ for all $u_1, u_2 \in U_\rho^\gamma(S)$.
		\end{enumerate}
	\end{proposition}

The space of continuous functions from $[0,T]$ to $\H$ is denoted by $\mathcal{C}([0,T];\H)$. We endow this space with the usual supremum norm, that is, for $f\in \mathcal{C}([0,T];\H)$, $$\|f\|_\infty = \sup_{t\in [0,T]}\|f(t)\|.$$
We set 
\begin{equation*}
\mathcal{C}_0([0,T];\H) = \{w\in \mathcal{C}([0,T];\H):w(0) = 0\}.
\end{equation*}
A family of functions $A\subset \mathcal{C}([0,T];\H)$ is said to be \emph{equicontinuous} if for all $\epsilon>0$, there is $\delta>0$ such that for all $s,t\in [0,T]$ with $|s-t|<\delta$ implies $\forall h\in A : \|h(s)-h(t)\|< \epsilon$.
    
	\begin{definition}
		A set-valued mapping $C\colon [0,T]\tto \H$ is said to be Hausdorff-continuous if, for every $\epsilon >0$, there exists $\delta>0$ such that 
        $$
        d_H(C(s),C(t))<\epsilon \textrm{ for all } s,t\in [0,T] \text{ with } |t-s|<\delta.$$
	\end{definition}
	For a set-valued mapping $C\colon [0,T]\tto \H$, define its modulus of continuity  by
    \begin{equation}\label{mod-cont-def}
        \mathscr{P}_C(r) := \sup\{ d_H(C(t),C(s)): t,s\in [0,T], |t-s|\leq r\}.
    \end{equation}
    If $C\colon [0,T]\to \H$ is Hausdorff-continuous, then $\displaystyle\lim_{r\searrow 0} \mathscr{P}_C(r) = 0$ and the map $r\mapsto\mathscr{P}_C(r)$ is nondecreasing, so we may define its (right) generalized inverse by 
    $$\mathscr{P}_C^{-1}(x) := \inf\{ r\geq 0 : \mathscr{P}_C(r)\geq x \},$$
    with the convention $\inf \emptyset = \infty$. Moreover, we observe that $\mathscr{P}_C(r)\geq x>0$ implies that $r\geq \mathscr{P}_C^{-1}(x)>0$. 

    \begin{lemma}\label{lem-equi-mc}
		Suppose that $A\subset \mathcal{C}([0,T];\H)$ is an equicontinuous set, and let $C$ be a set-valued mapping with closed values that is Hausdorff-continuous on $[0,T]$. For every $h\in A$, define $C_h := C-h$. Then $$\lim_{r\searrow 0}\sup_{h\in A}\mathscr{P}_{C_h}(r) = 0.$$
        Moreover, for all $\eta>0$, $\displaystyle\inf_{h\in A} \mathscr{P}_{C_h}^{-1}(\eta)>0$.
	\end{lemma}
	\begin{proof}
	Consider $\epsilon>0$. First, there exists $\delta_1>0$ such that for all $r\in ]0,\delta_1[$, we have $\mathscr{P}_C(r)<\epsilon/2$. Moreover, by equicontinuity, there exists $\delta_2>0$ such that for every $t,s\in [0,T]$ with $|t-s|<\delta_2$, $$\sup_{h\in A}\|h(t)-h(s)\|<\epsilon/2.$$
    Set $\delta:=\min\{\delta_1,\delta_2\}$. Then, for $r<\delta$ and $t,s$ with $|t-s|< r$, then
		\begin{equation*}
			d_H(C_h(t),C_h(s)) \leq d_H(C(t),C(s)) + \|h(t)-h(s)\|
			\leq \mathscr{P}_C(r) + \epsilon/2\leq \epsilon,
		\end{equation*}
		which proves that $\displaystyle\lim_{r\searrow 0}\sup_{h\in \mathcal{A}}\mathscr{P}_{C_h}(r) = 0$. Finally, suppose that for certain $\eta>0$, $\inf_{h\in A} \mathscr{P}_{C_h}^{-1}(\eta)=0$, then there is a sequence $(h_n)\subset A$ such that $\displaystyle\lim_{n\to \infty}\mathscr{P}_{C_{h_n}}^{-1}(\eta)=0$. By definition, there is $r_n\geq 0$ such that $\mathscr{P}_{C_{h_n}}^{-1}(\eta) + \frac{1}{n}\geq r_n$ where $\mathscr{P}_{C_{h_n}}(r_n)\geq \eta$, it follows that $\sup_{h\in A} \mathscr{P}_{C_{h}}(r_n)\geq \eta, \forall n\in \N$, but $r_n\to 0$, therefore $$\displaystyle\lim_{n\to\infty} \sup_{h\in A} \mathscr{P}_{C_{h}}(r_n) = 0$$ which is a contradiction. 
	\end{proof}
    Let us consider a function $f\colon [0,T]\to \H$. For every $a,b\in [0,T]$ with $a<b$,  we define the \emph{variation} of $f$ on $[a,b]$ as 
    $$
    \hspace{-1mm}\operatorname{var}(f;[a,b]) = \sup_{n\in \N}\left\{ \sum_{i=1}^n\|f(s_{i+1})-f(s_i)\| : a\leq s_1\leq \cdots\leq s_{n+1}\leq b \right\}.
    $$
    We say that $f$ has bounded variation on the interval $[a,b]$ if $\operatorname{var}(f;[a,b])$ is finite. The space of continuous functions from $[0,T]$ to $\H$ with bounded variation is denoted by $\mathcal{C}_{BV}([0,T];\H)$. For $x\in \mathcal{C}_{BV}([0,T];\H)$, we denote by \( dx \) its differential measure (see, e.g., \cite[Theorem 1, p. 358]{MR206190}) and recall that it is a nonatomic \( \H \)-valued measure such that 
    $$
    \int_{[s,t]} dx = x(t) - x(s)  \textrm{ for all }  s,t \in [0,T].
    $$
    The \textit{modulus measure} \( | m | \) of an \( \H \)-valued measure \( m \) is defined as 
    \[ |m|(A) := \sup\left\{ \sum_{B\in \pi} \|m(B)\| : \pi \text{ is a finite measurable partition of } A \right\}. \]
    If \( f\colon [0,T] \to \H \) is continuous and \( u\in \mathcal{C}_{BV}([0,T];\H) \), we set
    \begin{equation*}
    	\int_{[s,t]} \langle f , du\rangle := \int_{[s,t]} \left\langle f, \frac{du}{|du|}\right\rangle |du|,
    \end{equation*}
    where \( \frac{du}{|du|} \) is the density of \( du \) with respect to \( |du| \). Moreover, observe that \( \frac{du}{|du|}\in \mathbb{S} \) for \( |du| \)-almost every point.  When the measure is nonatomic, the notation \( \int_s^t \) refers to the integral on the interval $[s,t]$. The notation \( \nu \ll \mu \) means that \( \nu \) is absolutely continuous with respect to \( \mu \).	
	We state the following version of Grönwall's inequality.
	\begin{lemma}\label{gronwall-lemma}
		Let $\mu$ be a Borel measure on $[a,b]$, let $\epsilon>0$, and let $f\colon [a,b]\to\R$ be a right-continuous function satisfying that 
		\begin{equation*}
			0\leq f(t)\leq \epsilon + \int_{[a,t[}f(s)d\mu \quad \textrm{ for all } t\in [a,b].
		\end{equation*}
		Then,  $f(t)\leq \epsilon\exp\left(\mu([a,t[)\right)$ for all $t\in [a,b]$.
	\end{lemma} 
	\begin{proof}
		This is a consequence of \cite[Lemma 7.11]{MR4484114}.
	\end{proof}
    
The following lemma establishes a property of a uniformly prox-regular moving set with nonempty interior. The proof is given in \ref{appendix1}. 

    \begin{lemma}\label{lemma_partition}
Let $C\colon [0,T]\rightrightarrows \H$ be Hausdorff-continuous with closed values.
Assume that there exist $\rho>0$ such that, for every $t\in[0,T]$, the set $C(t)$ has nonempty interior and is $\rho$-uniformly prox-regular. Suppose that for some $\ov t\in [0,T]$ there is $r>0$ and $\bar x\in C(\ov t)$ such that $\mathbb{B}_r(\ov x)\subset C(\ov t)$, then, $\mathbb{B}_{r/2}(\ov x)\subset C(t)$ for all $t\in ]\ov t-\delta, \ov t+\delta[\cap [0,T]$ where $\delta := \mathscr{P}_{C}^{-1}(\min\{r/2,\rho/2\})$.
\end{lemma}

In the next lemma, we prove a property of uniformly prox-regular moving sets which are Hausdorff-continuous. Its proof is given in \ref{C(t)=H.proof}.
\begin{lemma}\label{C(t)H-statement} Suppose that $C\colon [0,T]\tto \H$ is a Hausdorff-continuous set-valued mapping with closed values. Assume that there exist $\rho>0$ such that, for every $t\in[0,T]$, the set $C(t)$ is $\rho$-uniformly prox-regular. Assume that for some $t_0\in [0,T]$, $C(t_0) = \H$, then for all $t\in [0,T]$, $C(t) = \H$. 
    \end{lemma}


	\section{Main Geometric Hypotheses}\label{sec3}
    In this section, we present the main hypotheses that will be assumed throughout the paper concerning the set-valued mapping $C\colon [0,T]\tto \H$.
    \begin{enumerate}[label=$(\textsf{H}_{\arabic*})$]
		\item \label{H1mgh}  For all $t\in [0,T]$, $C(t)$ is nonempty, closed and $\rho$-uniformly prox-regular for some $\rho>0$.
		\item \label{H2mgh}  The set-valued map $C$ is Hausdorff-continuous.
		\item \label{H3mgh}  For all $t\in [0,T]$, $\bd C(t)$ is compact.
		\item \label{H4mgh}  There exist $r,L>0$ such that for all $t\in [0,T]$ and $x\in\bd C(t)$, there exists $z(t,x)\in \H$ such that 
		\begin{equation*}
			\operatorname{co}(\{x\}\cup\mathbb{B}_{2r}(z(t,x)))\subset C(t)\text{ and }\|z(t,x)-x\|\leq Lr.
		\end{equation*} 
		\item \label{H5mgh} There exist $\delta>0$, $L> 0$ such that for all $t\in [0,T]$ and $x\in\bd C(t)$, there exists $\ell_x(t)\in\mathbb{S}$ satisfying
		\begin{equation*}
			\sup_{y\in\mathbb{B}_{\delta}(x)\cap C(t)}\sigma(\ell_x(t);N^P(C(t);y)\cap \mathbb{S})\leq -\frac{1}{L}.
		\end{equation*}
	\end{enumerate}
    The condition \ref{H4mgh} was introduced in \cite{MR1946545} to study the existence and uniqueness of the sweeping process dynamic. Furthermore, in \cite[Definition 3.1]{MR2812587} is introduced a similar notion to \ref{H5mgh} based on \cite[eq. (5), p. 514]{MR745330}, called admissibility, and it is possible to prove that \ref{H5mgh} is actually less restrictive (i.e., admissibility implies \ref{H5mgh}). It is important to remark that condition \ref{H5mgh} can be viewed as a time-dependent version of the geometric condition introduced in \cite[Condition (B), p. 456]{MR873889}.

\begin{proposition}\label{rmk-pr-cpr}
     \ref{H4mgh} holds provided that at least one of the following conditions is satisfied:
    \begin{itemize}
        \item[(i)] The set-valued mapping $C\colon [0,T]\tto \H$ satisfies \ref{H2mgh}, and for all $t\in [0,T]$, the set $C(t)$ is convex with nonempty interior and has bounded boundary.
        \item[(ii)] The Hilbert space $\H$ is finite dimensional and, for all $t\in [0,T]$,  the set $A(t):=\cl(\H\setminus C(t))$ is $\beta$-uniformly prox-regular for some $\beta>0$.
    \end{itemize}
\end{proposition}
\begin{proof}
Observe that (i) is a direct consequence of Lemma \ref{lemma_partition} together with an argument of compactness of $[0,T]$. To check (ii), fix any $t\in [0,T]$ and $x\in A(t)$. Given $\epsilon>0$, there exists $y_\epsilon\in \mathbb{B}_{\epsilon}(x)\cap \bd A(t)$ such that $N^P(A(t);y_\epsilon) \neq\{0\}$ (see \cite[Corollary 1.6.2]{Clarke1998}). Take $v_\epsilon\in N^P(A(t);y_\epsilon)\setminus \{0\}$. By assumption, for all $\gamma\in ]0,\beta[$ we have 
		\begin{equation*}
			\mathbb{B}_{\gamma}\left(y_\epsilon+\frac{\gamma}{\|v_\epsilon\|}v_\epsilon\right)\subset \H\setminus A(t) = \inte C(t).
		\end{equation*}
		It follows that $\operatorname{co}\big(\{y_\epsilon\}\cup \mathbb{B}_{\beta/4}\big(y_\epsilon+\frac{\beta}{2\|v_\epsilon\|}v_\epsilon\big)\big)\subset C(t)$. Up to a subsequence, we have $\frac{v_\epsilon}{\|v_\epsilon\|}\to v$ with $\|v\| = 1$ and $y_\epsilon\to x$. It is then not difficult to see that $$\operatorname{co}\left(\{x\}\cup \mathbb{B}_{\beta/8}\left(x+\frac{\beta}{2}v\right)\right)\subset C(t),$$
		which proves (ii) by taking $r = \beta/8$ and $L=4$.
\end{proof}
    \begin{remark} 
    It is worth noting that there exist set-valued mappings with uniformly prox-regular values that do not satisfy assumption \ref{H4mgh}, even when they have nonempty interior. Indeed, consider $C(t)\equiv C$, where 
    $$C:= \{(x,y)\in\R^2 : (x+1)^2 + (y-1)^2\geq 1, x\in [-1,0], y\in [0,1]\}.
    $$
		\begin{figure}[h]
			\centering
			\captionsetup{justification=centering}
			\begin{tikzpicture}
				\filldraw[color=gray!20,draw=black];
				\draw (-2,0) arc (270:360:2cm);
				\draw (-2,0) -- (0,0);
				\draw (0,0) -- (0,2);
				\draw (-0.3,0.5) node[anchor=north east] {$C$};
				\draw (-2,0) node[left] {$(-1,0)$};
				\draw (0,0) node[right] {$(0,0)$};
				\draw (0,2) node[right] {$(0,1)$};
			\end{tikzpicture}
			\caption{A uniformly prox-regular set that does not satisfy neither \ref{H4mgh} nor \ref{H5mgh}.}
			\label{fig:enter-label}
		\end{figure}
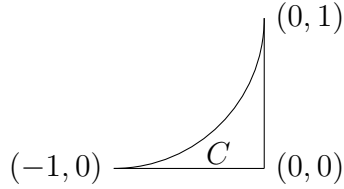
It is clear that $t\tto C(t)$ is Hausdorff-continuous and takes uniformly prox-regular values. However, at corner points it is not possible to find a ball $\mathscr{B}\subset C$ such that $\operatorname{co}(\{x_0\}\cup \mathscr{B})\subset C$ (see Figure \ref{fig:enter-label}). We also note that $C$ does not satisfy condition \ref{H5mgh}. 

\noindent On the other hand, the boundedness assumption in Proposition \ref{rmk-pr-cpr} (i) is indeed necessary: there exist unbounded convex sets with nonempty interior that do not satisfy \ref{H4mgh}. For example, let $A = \{ (x,y,z)\in \mathbb{R}^3 : -1\leq y\leq 1, x=1 ,z\geq 0\}$ and $B = \{ (x,y,z)\in \mathbb{R}^3 : y=0, x+z=0, z\geq 0 \}$. Then, the convex hull $S:= \ov{\text{co}}(A\cup B)$ fails to satisfy \ref{H4mgh}.
	\end{remark}

	\section{Relationships among the Geometric Assumptions}\label{sec4}

    In this section, we examine the relationships among the geometric hypotheses introduced in the previous section. First, we show that the Hausdorff continuity of a set-valued mapping taking uniformly prox-regular values implies the continuity of its boundary. The proof can be found in \ref{appendix2}.
	\begin{proposition}\label{h1h2-bd}
		Suppose that $C\colon [0,T]\tto \H$ is a set-valued mapping satisfying \ref{H1mgh} and  \ref{H2mgh}. Then the set-valued mapping $t\tto \bd C(t)$ is Hausdorff-continuous on its domain.
	\end{proposition}

The following lemma establishes that property \ref{H4mgh} is preserved under perturbation by continuous functions.
	\begin{lemma}\label{prop-preser-cm}
        Suppose $C\colon [0,T]\tto \H$ satisfies \ref{H4mgh}. For any function $h\colon [0,T]\to \H$, the shifted moving set $C_h := C-h$ also satisfies \ref{H4mgh}.
	\end{lemma}
	\begin{proof}
	Take $r,L>0$ as given by \ref{H4mgh}. If $x\in \bd C_h(t)$, then $x+h(t)\in\bd C(t)$. Hence, there exists $z\in\H$ such that $\|z-x-h(t)\|\leq rL$ and $$\operatorname{co}(\{x+h(t)\}\cup \mathbb{B}_{2r}(z))\subset C(t).$$ Setting $z':= z-h(t)$, we obtain $\|z'-x\|\leq rL$ and $\operatorname{co}(\{x\}\cup\mathbb{B}_{2r}(z'))\subset C_h(t)$, which completes the proof.
	\end{proof}
    The next result establishes that in general the condition \ref{H4mgh} is weaker than \ref{H5mgh}. 
    
    \begin{proposition}\label{H4-is-weaker-than-H5}
        Let us consider a moving set $C\colon [0,T]\tto \H$ satisfying \ref{H1mgh}, \ref{H2mgh} and \ref{H5mgh}. Then, \ref{H4mgh} holds.
    \end{proposition}

    \begin{proof} Suppose that \ref{H5mgh} holds with constants $\delta, L>0$. Take any $r_0\in ]0,\gamma[$ where
    \[
\gamma:=
\min\left\{
\tfrac{\rho}{1+\frac{1}{4L}},
\;
\tfrac{\delta}{2\left(1+\frac{1}{4L}\right)},
\;
\tfrac{3\rho}{8L\left(1+\frac{1}{4L}\right)^2}
\right\},
\]
and define
\[
\tilde r=\frac{r_0}{8L},
\qquad
\tilde L=8L.
\]
We aim to prove that \ref{H4mgh} holds with constants $\tilde{r}, \tilde{L}$. Indeed, fix $t\in [0,T]$ and $x\in \bd C(t)$, then, in particular there is $\ell_x(t)\in \mathbb{S}$ such that
\begin{equation}\label{propH5-part} 
    \forall y\in \B_\delta(x)\cap C(t),\quad
\forall n\in N^P(C(t);y)\cap \mathbb{S},
\quad \langle n,\ell_x(t)\rangle \le -\frac{1}{L}.
\end{equation}
It is enough to prove that
\begin{equation}\label{inclusionH4}
    \overline{\text{co}}(\{x\}\cup \mathbb{B}_{2\tilde{r}}(z(t,x)))\subset C(t)
\end{equation}
where $z(t,x) := x+r_0\ell_x(t)$. Note that $\|z(t,x)-x\| = r_0 = \tilde{r}\tilde{L}$. Suppose that \eqref{inclusionH4} does not hold, then there is $u\in\text{co}(\{x\}\cup \mathbb{B}_{2\tilde{r}}(z(t,x)))$ such that $u\notin C(t)$. Note that there are $\lambda\in ]0,1]$ and $v\in 2\tilde{r}\mathbb{B}$ such that $$u = x+\lambda r_0\ell_x(t) + \lambda v.$$
Since $u\notin C(t)$ we have $d:=d(u;C(t))>0$. Moreover,
\begin{equation*}
    d\leq \|u-x\| = \|\lambda r_0\ell_x(t) + \lambda v\|\leq \lambda(r_0 + 2\tilde{r}) = \lambda\left(1+\frac{1}{4L}\right)r_0<\rho.
\end{equation*}
Then, by \ref{H2mgh} and Proposition \ref{prox_reg_prop} we have $p:=\text{proj}_{C(t)}(u)$ is well-defined. Note that $d = \|u-p\|>0$. Moreover,
\begin{equation*}
    \|p-x\|\leq \|p-u\| + \|u-x\| = d+\|u-x\|\leq 2\|u-x\|\leq 2\lambda\left(1+\frac{1}{4L}\right)r_0<\delta.
\end{equation*}
It follows that $p\in \mathbb{B}_\delta(x)$. Thus, we consider $n:= \frac{u-p}{\|u-p\|}$, and it is clear that $n\in N^P(C(t);p)$, then by \eqref{propH5-part} we have $\langle n,\ell_x(t)\rangle\leq \frac{-1}{L}$.
The uniform prox-regularity implies that 
\begin{equation}\label{bound1-H4}
    \langle n,x-p\rangle\leq \frac{1}{2\rho}\|x-p\|^2\leq \frac{2\lambda^2}{\rho}\left(1+\frac{1}{4L}\right)^2r_0^2<\frac{3\lambda^2 r_0}{4L}
\end{equation}
On the other hand,
\begin{equation}\label{bound2-H4}
    \begin{aligned}
    \langle n,x-p\rangle = \langle n,x-u\rangle + \langle n, u-p\rangle&= \langle n,-\lambda r_0\ell_x(t) - \lambda v\rangle + d\\
    &\geq \frac{\lambda r_0}{L} - \lambda\langle n,v\rangle + d\\
    &\geq \frac{\lambda r_0}{L} - \lambda\frac{r_0}{4L} + d = \frac{3\lambda r_0}{4L} + d.
    \end{aligned}
\end{equation}
From \eqref{bound1-H4} and \eqref{bound2-H4}, it follows that $d< \frac{3\lambda^2 r_0}{4L}-\frac{3\lambda r_0}{4L} = \frac{3\lambda r_0}{4L}(\lambda-1)\leq 0$ which contradicts that $d>0$. It follows that \eqref{inclusionH4} holds, concluding the desired.
    \end{proof}

	Now, we are going to study the converse of Proposition \ref{H4-is-weaker-than-H5}.  The following lemmas will be useful in this line.

\begin{lemma}\label{lemma_col_mon_1}
		Let $C$ be a closed set and $x\in\bd C$. Suppose there exist $r>0$, $L>0$, and $z\in\H$ such that $\operatorname{co}(\{x\}\cup \mathbb{B}_{2r}(z))\subset C$ and $\|z-x\|\leq Lr$. Then,
		\begin{equation*}
	 \left\langle v,\frac{z-x}{\|z-x\|} \right\rangle\leq -\frac{\|v\|}{L} \quad \textrm{ for all } v\in N^P(C;x).
		\end{equation*}
	\end{lemma}
	\begin{proof}
		Take $v\in N^P(C;x)$. If $v=0$, then the inequality holds trivially. Assume $v\neq 0$. Since $A:=\overline{\operatorname{co}}(\{x\}\cup \mathbb{B}_{2r}(z))\subset C$, we have $v\in N^P(A;x)$. As $A$ is convex and $z+\frac{r}{\|v\|}v\in A$, it follows that
		\begin{equation*}
			\left\langle v,z+ \frac{r}{\|v\|}v - x \right\rangle\leq 0 \quad \Rightarrow \quad r\|v\| + \langle v,z-x\rangle\leq 0.
		\end{equation*}
		Since $x\in\bd C$, we have $z\neq x$, and moreover $\|x-z\|\leq Lr$. Dividing the inequality above by $\|x-z\|$ yields the desired estimate, completing the proof.
	\end{proof}
	
	We recall that if $x\in C$, then $\partial_P d_C(x)\cap \mathbb{S} = N^P(C;x)\cap \mathbb{S}$ (see \eqref{normal-distancia}). The next lemma establishes the upper semicontinuity of the map $(t,x)\mapsto \sigma(\xi;\partial_P d_{C(t)}(x)\cap \mathbb{S})$ for all $\xi\in \H$ (a similar result can be found in \cite[Lemma 2]{MR4822735}). 
	\begin{lemma}\label{upper-semi-support}
	Suppose that $\H$ is finite dimensional, and let $C\colon [0,T]\tto \H$ be a set-valued mapping satisfying \ref{H1mgh} and \ref{H2mgh}. Then, for all $\ov \xi\in \H$, $\ov t\in [0,T]$, and $\ov x\in C(\ov t)$,
		\begin{equation*}
			\limsup_{\substack{x\to \ov x, t\to \ov t\\ \xi\to \ov \xi}} \sigma(\xi;\partial_P d_{C(t)}(x)\cap \mathbb{S}) \leq \sigma(\ov \xi;\partial_P d_{C(\ov t)}(\ov x)\cap \mathbb{S}).
		\end{equation*}
	\end{lemma}
    \begin{proof}
        See \ref{appendix4}.
    \end{proof}

The following proposition provides a sufficient condition for \ref{H5mgh} to hold.
\begin{proposition}\label{lemma-H5-suficiente}
    Suppose that $\H$  is finite-dimensional. Let $C\colon [0,T]\tto \H$ be a set-valued mapping  satisfying \ref{H1mgh}, \ref{H2mgh}, and \ref{H3mgh}. Assume there exists $\gamma>0$ such that, for all $t\in [0,T]$ and $x\in \bd C(t)$, there is $\ell\in \mathbb{S}$ with
    \begin{equation}\label{eqn-support-lemma}
        \sigma(\ell;N^P(C(t);x)\cap \mathbb{S})\leq -\gamma.
    \end{equation}
    Then \ref{H5mgh} holds.
\end{proposition}
\begin{proof}
If for some $\bar t\in [0,T]$, $C(\bar t) = \H$, then $C(t) = \H, \forall t\in [0,T]$ (see Lemma \ref{C(t)H-statement}), then in this case \ref{H5mgh} holds trivially. From now on, we assume that $\forall t\in [0,T] : C(t)\neq \H$ and therefore $\forall t\in [0,T] : \bd C(t)\neq \emptyset$. Suppose, on the contrary, that \ref{H5mgh} does not hold. Then, for every $n\in\N$, there exist $t_n\in [0,T]$ and $x_n\in \bd C(t_n)$ such that, for all $\ell\in \mathbb{S}$, there is $y_n^{\ell}\in \mathbb{B}_{1/n}(x_n)\cap C(t_n)$, satisfying
		\begin{equation*}
\sigma(\ell;N^P(C(t_n);y_n^{\ell})\cap\mathbb{S})\geq -\frac{\gamma}{2}.
		\end{equation*}
Passing to a subsequence, we obtain $t_n\to \bar t$ and $x_n\to \bar x\in \bd C(\bar t)$, since the mapping $t\tto \bd C(t)$ is Hausdorff continuous on $[0,T]$ with compact values (see Proposition \ref{h1h2-bd}). Then, by \eqref{eqn-support-lemma}, there exists $\bar\ell\in \mathbb{S}$ such that 
		\begin{equation*}
			\sigma(\bar\ell;N^P(C(\bar t);\bar x)\cap\mathbb{S})\leq -\gamma.    
		\end{equation*} 
Moreover, we have  $y_n^{\bar\ell}\to \bar x$. Using the  above inequality together with Lemma \ref{upper-semi-support}, we obtain 
		\begin{align*}
				-\frac{\gamma}{2} &\leq \liminf_{n\to\infty}\sigma(\bar\ell; \partial_P d_{C(t_n)}(y_n^{\bar \ell})\cap\mathbb{S})\\
				&\leq \limsup_{y\to \bar x, s\to \bar t} \sigma(\bar\ell;\partial_P d_{C(s)}(y)\cap \mathbb{S})\\
				&\leq \sigma(\bar\ell;\partial_P d_{C(\bar t)}(\bar x)\cap\mathbb{S})\\
                &= \sigma(\bar\ell;N^P(C(\bar t);\bar x)\cap\mathbb{S})\\
                &\leq -\gamma.
		\end{align*} 
        This yields a contradiction.
\end{proof}

The following proposition shows that \ref{H4mgh} implies \ref{H5mgh} as a direct consequence of the preceding criterion.
	\begin{proposition}\label{prop_col_lions}
		Suppose that $\H$ has finite dimension. Consider a set-valued mapping $C\colon [0,T]\tto \H$ satisfying \ref{H1mgh}, \ref{H2mgh}, \ref{H3mgh}, and \ref{H4mgh}. Then \ref{H5mgh} holds.
	\end{proposition}
	\begin{proof}
		Suppose that \ref{H4mgh} holds with constants $L,r>0$. By Lemma \ref{lemma_col_mon_1}, we have that for all $t\in [0,T]$ and $x\in \bd C(t)$, there exists $\ell\in \mathbb{S}$ such that 
        $$\sigma(\ell;N^P(C(t);x)\cap \mathbb{S})\leq -\frac{1}{L}.
        $$
        Thus \eqref{eqn-support-lemma} holds. Therefore, by Proposition \ref{lemma-H5-suficiente}, condition \ref{H5mgh} is satisfied.
	\end{proof}

    We now examine moving sets defined as finite intersections of sublevel sets of smooth functions. In many applications, moving sets are described in this way, so it is worth studying their properties. In \cite{MR3458154}, general conditions are given to ensure the uniform prox-regularity of such sets. Moreover, the crowd motion model can be formulated as a sweeping process with a moving set of this form (see, e.g., \cite{MR2800713, MR2812587, MR866273}).  
    
    Suppose that $\H$ is finite-dimensional. Consider functions $g_i\colon \H\to \R$ for $i\in I := \{1,...,m\}$, the moving sets $C_i(t) := \{x\in \H : g_i(t,x)\leq 0\}$, and 
    \begin{equation*}
        C(t) := \bigcap_{i=1}^m C_i(t).    
    \end{equation*}
    For every $(t,x)\in \gph C$, we define $I(t,x) := \{i\in I : g_i(t,x) = 0\}$, and for $\epsilon>0$, $$I_\epsilon(t,x) := \{i\in I : -\epsilon\leq g_i(t,x)\leq 0\}.$$
    Regarding the functions $g_i$, we impose the following hypotheses:
 \begin{enumerate}[label = ($\alph*$)]
    \item There exists $\epsilon>0$ such that $\mathscr{Q}_\epsilon := \{(t,x)\in \gph C : I_{\epsilon}(t,x)\neq\emptyset\}$ is bounded. Moreover, for every $(t,x)\in \mathscr{Q}_\epsilon$, the set $\{\nabla g_i(t,x) : i\in I_{\epsilon}(t,x)\}$ is positive linearly independent, i.e.,  
    $$
    \textrm{if } \sum_{i\in I_\epsilon(t,x)}\alpha_i \nabla g_i(t,x) = 0 \textrm{ with } \alpha_i\geq 0, \textrm{ then } \alpha_i = 0 \textrm{ for every } i\in I_\epsilon (t,x).
    $$
    \item There is $\eta'>0$ such that for all $i\in I$,  $g_i$ and $\nabla g_i$ are continuous, and the family $$\{ t\mapsto g_i(t,x) : x\in \mathsf{D} + \eta'\B \}$$ is equicontinuous where $\mathsf{D} := \bigcup_{s\in [0,T]} C(s)$.
 	\item There exist $L>0$ and $\eta>0$ such that for all $t\in [0,T]$ and $i\in I$, the mapping $\nabla g_i(t,\cdot)$ is $L$-Lipschitz on $C(t) + \eta\B$.
 \end{enumerate}
 \begin{proposition}\label{prop-sublevel-h4-h5}
 	Under the above assumptions, the following statements hold for the set-valued mapping $C$:
 	\begin{enumerate}[label = (\roman*)]
 		\item $C$ satisfies \ref{H4mgh}.
 		\item $C$ is Hausdorff continuous.
 		\item There exists $\rho>0$ such that, for all $t\in [0,T]$, the set $C(t)$ is $\rho$-uniformly prox-regular.
 	\end{enumerate}
 \end{proposition}
\begin{proof} 
See \ref{appendix3}.
 \end{proof}
	
\section{Deterministic Sweeping Process with Continuous Moving  Sets}\label{sec5}
Let $\mathcal H$ be a Hilbert space and let $C\colon [0,T]\rightrightarrows \mathcal H$ be a set-valued map. Since, in general, one cannot expect solutions of the sweeping process to be absolutely continuous unless the mapping $t\tto C(t)$ has sufficient regularity (e.g., is absolutely continuous or Lipschitz continuous), we adopt Moreau's measure-theoretic formulation (see \cite{MR508661}). In this setting, the ``derivative"  of a trajectory $x$ is understood as the Radon-Nikodym density of the vector measure $dx$ with respect to its total variation measure $\vert dx\vert$. 
\begin{definition}
Given $x_0\in C(0)$, a function $x\colon [0,T]\to \H$ is a solution of the sweeping process driven by $C$ with initial condition $x_0$ if 
\begin{enumerate}
  \item[(i)] $x\in \mathcal{C}_{BV}([0,T];\H)$ and $x(t)\in C(t)$ for every $t\in[0,T]$;
  \item[(ii)] $x(0)=x_0$;
  \item[(iii)] the vector measure $dx$ is absolutely continuous with respect  $|dx|$ and
        \begin{equation}\tag{$\mathsf{SP}$}\label{sp1}
        \left\{
        \begin{aligned}
          \frac{dx}{|dx|}(t) &\in\ -N^{P}\!\big(C(t);x(t)\big)
          \quad\text{ for } |dx|\text{-a.e. } t\in[0,T],\\
          x(0)&=x_0\in C(0).
          \end{aligned} 
          \right.
        \end{equation}
\end{enumerate}
\end{definition}
    
In \cite{MR1946545}, Colombo and Monteiro-Marques studied existence and uniqueness for the sweeping process \eqref{sp1} in the setting where $C$ is continuous with respect to the Hausdorff distance, each set $C(t)$ is $\rho$-uniformly prox-regular, and $x_0$ is given. Because the moving set may fail to have finite variation, they impose additional geometric assumptions.
	
\begin{theorem}\label{mc-thm}
Assume $C\colon [0,T]\tto \H$ satisfies \ref{H1mgh}, \ref{H2mgh}, and \ref{H4mgh}. If $t\tto\bd C(t)$ is Hausdorff continuous, then \eqref{sp1} admits a unique solution.
\end{theorem}
By Proposition \ref{h1h2-bd}, the continuity assumption on $t\tto \bd C(t)$ in Theorem \ref{mc-thm} is superfluous. Recently, in \cite{recupero_stra_2025_excess}, Recupero and Stra established the same simplification in an even more general setting, under the sole assumption that \(C\) is Hausdorff right-continuous. Their argument relies on a different technique and yields a more general result. The following well-posedness result improves \cite[Theorem 4.2]{MR1946545}.
	\begin{corollary}\label{mc-ch-cor}
Assume $C\colon [0,T]\tto \H$ satisfies \ref{H1mgh}, \ref{H2mgh},  and \ref{H4mgh}. Then, for every $h\in \mathcal{C}_0([0,T];\H)$, there exists a  unique function $x_h\in \mathcal{C}_{BV}([0,T];\H)$ solving
\begin{equation}\tag{$\mathsf{SP_h}$}\label{sp1h}
\left\{
\begin{aligned}
\frac{dx_h}{|dx_h|}(t)&\in -N^P(C(t)-h(t);x_h(t)) & \textrm{ for } |dx_h|\textrm{-a.e. } t\in [0,T],\\
				x_h(0)&=x_0\in C(0).
\end{aligned}
\right.
\end{equation}
\end{corollary}
\begin{proof} Note that $C_h := C-h$ still satisfies \ref{H1mgh}, \ref{H2mgh}, and \ref{H4mgh} by Lemma \ref{prop-preser-cm}. Then, apply Theorem \ref{mc-thm}.
	\end{proof}
    
Moreover, \cite{MR1946545} provides estimates for the total variation of solutions to \eqref{sp1}.
	\begin{proposition}\label{prop_var_cm}
Let $C\colon [0,T]\tto \H$ satisfying \ref{H1mgh} and \ref{H2mgh}. Suppose there exist $r>0$ and $z\in\H$ such that $\mathbb{B}_{r}(z)\subset C(t)$ for all $t\in [0,T]$ and $\frac{1}{\rho}(\|x_0 - z\|+r)^2< r$. Then \eqref{sp1} has an unique solution $x\in \mathcal{C}_{BV}([0,T];\H)$  with
\begin{equation*}
\operatorname{var}(x;[0,T])\leq \frac{\|x_0-z\|^2}{2(r- \frac{d^2}{\rho})}, \quad  \textrm{ where } d = \|x_0-z\| + r.
\end{equation*}
	\end{proposition}
    \begin{proof}
This follows directly from \cite[Lemma 2.2]{MR1946545} and \cite[Theorem 0.2.1]{MR1231975}, using the convergence scheme developed in \cite[Theorem 4.1]{MR1946545}. 
    \end{proof}
    
Our goal is to obtain uniform bounds on the total variation of solutions under certain perturbations of the set-valued mapping $C$. In Lemma \ref{lem_eq_fam} and Proposition \ref{prop-bv-mc} below, we assume that $C$ satisfies \ref{H1mgh}, \ref{H2mgh}, and \ref{H4mgh}.
\begin{lemma}\label{lem_eq_fam}
Let $A\subset \mathcal{C}_0([0,T];\H)$ be equicontinuous. Then there exist $\delta>0$ and $\mathfrak{r}>0$ such that, for all $h\in A$ and all $(\ov t, \ov x)\in\gph \bd(C-h)$,  there exists $z\in \H$ with the property that, for every $t\in ]\ov t-\delta,\ov t+\delta[\cap [0,T]$, 
$$\mathbb{B}_{\mathfrak{r}}(z)\subset C(t)-h(t) \textrm{ and } \frac{1}{\rho}(\|\overline{x}-z\|+\mathfrak{r})^2<\frac{\mathfrak{r}}{2}.
$$
\end{lemma}
\begin{proof}
Assume \ref{H4mgh} holds with constants $r,L>0$. Let $\lambda := \min\{1,\frac{\rho}{ 2r(L+2)^2}\}$. Fix $h\in A$ and $\ov t\in [0,T]$ with $\overline{x}\in \bd(C(\ov t) - h(\ov t)) = -h(\ov t)+\bd C(\ov t)$. By \ref{H4mgh}, there exists $z(\ov{t},\ov{x})\in\H$ such that $\|z(\ov{t},\ov{x})-\ov x - h(\ov t)\|\leq Lr$ and $\operatorname{co}(\{\ov x + h(\ov t)\}\cup \mathbb{B}_{2r}(z(\ov{t},\ov{x})))\subset C(\ov t)$. Define $ z_\lambda(\ov t,\ov x) := \ov x + h(\ov t) + \lambda(z(\ov t,\ov x) - \ov x - h(\ov t))$. Then $\mathbb{B}_{2\lambda r}(z_{\lambda}(\ov t,\ov x))\subset C(\ov t)$ and also define $\mathsf{z} := z_\lambda(\ov t,\ov x) - h(\ov t)$. Note that
		\begin{equation*}
			\begin{aligned}
				\frac{1}{\rho}\left(\|\mathsf{z} - \overline{x}\| + \lambda r\right)^2&=\frac{1}{\rho}\left(\|z_{\lambda}(\ov t,\ov x) - \overline{x}-h(\ov t)\| + \lambda r\right)^2\\
                &< \frac{1}{\rho}\left((L+1)r\lambda + \lambda r\right)^2\\
				& = \frac{\lambda^2 r^2}{\rho }(L+2)^2\\
			&\leq \lambda\cdot \frac{r^2}{\rho }(L+2)^2\cdot \frac{\rho}{ 2r (L + 2)^2} = \frac{\lambda r}{2}.
			\end{aligned}
	\end{equation*}
We take $\mathfrak{r} = \lambda r$. Note that $\mathbb{B}_{2\mathfrak{r}}(\mathsf{z})\subset C(\bar t)-h(\bar t)$. Now, define $C_h := C-h$ and take $\delta_h := \mathscr{P}_{C_h}^{-1}(\min\{\mathfrak{r},\rho/2\})$, by virtue of Lemma \ref{lemma_partition} we have that
\begin{equation*}
    \forall t\in ]t-\delta_h, t+\delta_h[\cap [0,T] : \mathbb{B}_{\mathfrak{r}}(\mathsf{z})\subset C_h(t).
\end{equation*}
Observe that due to Lemma \ref{lem-equi-mc} we have that $\delta:=\inf_{h\in A} \delta_h>0$. It follows that 
\begin{equation*}
    \forall h\in A, \forall t\in ]t-\delta, t+\delta[\cap [0,T] : \mathbb{B}_{\mathfrak{r}}(\mathsf{z})\subset C_h(t),
\end{equation*}
concluding the desired.
	\end{proof}
    
	From the lemma above we deduce a uniform total variation bound for the solutions of \eqref{sp1h}.
	\begin{proposition}\label{prop-bv-mc}
		Let $A\subset \mathcal{C}_0([0,T];\H)$ be equicontinuous. For each $h\in A$, let $x_h\colon [0,T]\to \H$ solve \eqref{sp1h}. Then $$\sup_{h\in A}\operatorname{var}(x_h;[0,T])<\infty.$$
	\end{proposition}
\begin{proof}
Fix $h\in A$ and set $C_h := C-h$. Let $\delta,\mathfrak{r}>0$ be those given by Lemma \ref{lem_eq_fam}. For any $t\in [0,T[$ with $x_h(t)\in\bd C_h(t)$, there exists $z\in \H$ such that $\mathbb{B}_{\mathfrak{r}}(z)\subset C_h(s)$ and $\frac{1}{\rho}(\|z-x_h(t)\| + \mathfrak{r})^2<\frac{\mathfrak{r}}{2}$ for all $s\in ]t-\delta,t+\delta[\cap [0,T]$.  Proposition \ref{prop_var_cm} then yields 
\begin{equation}\label{var_boundary}
\operatorname{var}(x_h;[t,t+\delta])\leq \frac{\|x_h(t)-z\|^2}{2(\mathfrak{r}-\frac{1}{\rho} (\|z-x_h(t)\|+\mathfrak{r})^2)}\leq \frac{\|x_h(t)-z\|^2}{\mathfrak{r}}\leq \frac{\rho}{2}.
	\end{equation}
Let $I := \{t\in [0,T]: x_h(t)\in \inte C_h(t)\}$, which is an open set of $[0,T]$. For $t\in I$, define $\ell(t) = \inf\{s\geq t : x_h(s)\in \bd C_h(s)\}$. Then $\ell(t)>t$ and $\operatorname{var}(x_h;[t,\ell(t)]) = 0$. Set $t_0 := 0$. If $x_0\in I$, define  $t_1:= \min\{T,\ell(t_0)\}$ (otherwise, set $t_1 := t_0$). Since $x_h(t_1)\in\bd C_h(t_1)$,  set $t_2 := \min\{T,t_1 + \delta\}$. If $t_2\in I$, define $t_3 := \min\{T,\ell(t_2)\}$ (otherwise, set $t_3 := t_2$), and then define $t_4 := \min\{T,t_3 + \delta\}$. Proceeding inductively, let $n := \inf\{k\in\N : t_k = T\}$; by the construction $n<\infty$. Let $J:= \{j\leq n-1: x_h(t_j)\in \bd C_h(t_j)\}$. Then $\text{card}(J)\leq \lfloor T/\delta\rfloor + 1$ (where $\text{card}(J)$ denotes the cardinality of $J$). Therefore,
		\begin{equation*}
\operatorname{var}(x_h;[0,T])= \sum_{j\in J} \operatorname{var}(x_h;[t_j,t_{j+1}])\leq \text{card}(J)\frac{\rho}{2}\leq \frac{\rho}{2}(\lfloor T/\delta \rfloor + 1),
		\end{equation*}
where we used \eqref{var_boundary}. Consequently,  $\sup_{h\in A} \operatorname{var}(x_h;[0,T])\leq \frac{\rho}{2}(\lfloor T/\delta \rfloor + 1)$, which proves the result.
	\end{proof}

\begin{remark}\label{rmk-x0-independent}
It is worth noting that the estimate for $\sup_{h\in A} \operatorname{var}(x_h;[0,T])$ obtained in Proposition \ref{prop-bv-mc} does not depend on $x_0\in C(0)$. In fact, we have 
$$\sup_{x_0\in C(0)}\sup_{h\in A} \operatorname{var}(x_h(\cdot;x_0);[0,T])<\infty,$$
where $x_h(\cdot;x_0)$ denotes the solution of \eqref{sp1h} starting from $x_0$. 
\end{remark}
\begin{proposition}\label{lemma-cont-sol}
Suppose that $C\colon [0,T]\tto \H$ is a set-valued mapping satisfying \ref{H1mgh}, \ref{H2mgh} and \ref{H4mgh}. Let $A\subset\mathcal{C}_0([0,T];\H)$ be an equicontinuous set of functions. Then,
\begin{enumerate} 
\item[(a)] There is a constant $\mathscr{C}>0$ such that, for all $x_0,y_0\in C(0)$ and $u,v\in A$, 
\begin{equation*}
\|x_u(\cdot;x_0)-x_v(\cdot;y_0)\|^2_\infty\leq \mathscr{C}\left(\|x_0-y_0\|^2 + \|u-v\|_\infty + \|u-v\|_\infty^2\right).
\end{equation*}
\item[(b)] Let $(u_n)\subset \mathcal{C}_0([0,T];\H)$ be a sequence converging to $u\in\mathcal{C}_0([0,T];\H)$ uniformly, and let $(x_0^n)\subset C(0)$ converging to $x_0\in C(0)$. Then $x_{u_n}(\cdot;x_0^n)\to x_u(\cdot;x_0)$ uniformly on $[0,T]$. 
		\end{enumerate}
Consequently, the operator $\mathscr{F}\colon \mathcal{C}_0([0,T];\H)\times C(0)\to \mathcal{C}_{BV}([0,T];\H)$ defined by $\mathscr{F}(u,x_0) := x_u(\cdot;x_0)$ is continuous. 
	\end{proposition}
\begin{proof}
By Proposition \ref{prop-bv-mc} and Remark \ref{rmk-x0-independent}, we have $$
\displaystyle\mathscr{V} := \sup_{x_0\in C(0)}\sup_{h\in A}\operatorname{var}(x_h(\cdot;x_0);[0,T])<\infty.
$$
(a) Take $u,v\in A$ and $x_0,y_0\in C(0)$. Consider the solution $x_u(\cdot;x_0)$ of $(\mathsf{SP}_{u})$ starting from $x_0$ and the solution $x_v(\cdot;y_0)$ of  $(\mathsf{SP}_{v})$ starting from $y_0$. For brevity, we write $x_u := x_u(\cdot;x_0)$ and $x_v := x_v(\cdot;y_0)$. Define the finite measure $\mu := |dx_u| + |dx_v|$. Observe that $|dx_u|$-a.e. $-\frac{dx_u}{|dx_u|}(t)\in N^P(C(t)-u(t);x_u(t))$, due to the uniform prox-regularity of $C(t)$ and the fact that for all $t\in [0,T]$, $x_v(t) \in C(t)-v(t)$ one has $|dx_u|$-a.e.
        \begin{equation*}
				\left\langle -\frac{dx_u}{|dx_u|}(t),x_v(t)+v(t)-( x_u(t) + u(t))\right\rangle 
				\leq \frac{1}{2\rho}\|x_v(t)+v(t)-( x_u(t) + u(t))\|^2
		\end{equation*}
Since $|dx_u|\ll \mu$, by Radon-Nikodym Theorem, we have 
    \begin{equation*}
			\begin{aligned}
				&\int_{0}^t \left\langle -\frac{dx_u}{d\mu}(t),x_v(t)+v(t) - (x_u(t) + u(t)) \right\rangle 
				d\mu \\
				&\quad \quad\quad \quad\quad \quad\quad \quad\quad \quad\quad \quad \leq  \ \frac{1}{2\rho}\int_{0}^t\|x_u(t)+u(t) - (x_v(t) + v(t))\|^2  |dx_u|.
			\end{aligned}
		\end{equation*}
        Similarly, we obtain
        \begin{equation*}
			\begin{aligned}
				&\int_{0}^t \left\langle -\frac{dx_v}{d\mu}(t),x_u(t)+u(t) - (x_v(t) + v(t)) \right\rangle 
				d\mu \\
				&\quad \quad\quad \quad\quad \quad\quad \quad\quad \quad\quad \quad \leq  \ \frac{1}{2\rho}\int_{0}^t\|x_u(t)+u(t) - (x_v(t) + v(t))\|^2  |dx_v|.
			\end{aligned}
		\end{equation*}
        Summing the last two inequalities we get
\begin{equation*}
			\begin{aligned}
				&\int_{0}^t \left\langle -\frac{dx_u}{d\mu}(t)+\frac{dx_v}{d\mu}(t),x_v(t)+v(t) - (x_u(t) + u(t)) \right\rangle 
				d\mu \\
				&\quad \quad\quad \quad\quad \quad\quad \quad\quad \quad\quad \quad \leq  \ \frac{1}{2\rho}\int_{0}^t\|x_u(t)+u(t) - (x_v(t) + v(t))\|^2  d\mu.
			\end{aligned}
		\end{equation*}
This implies that, for all $t\in [0,T]$,
		\begin{equation*}
			\begin{aligned}  
\frac{1}{2}\|x_u(t)-x_v(t)\|^2\leq & \ \frac{1}{2}\|x_0-y_0\|^2 + \int_{0}^{t}  \left\langle -\frac{dx_u}{d\mu}(s) + \frac{dx_v}{d\mu}(s),u(s) -v(s) \right\rangle d\mu\\
				&+\frac{1}{2\rho}\int_{0}^t \|x_u(s)-x_v(s) + (u(s)-v(s))\|^2d\mu\\
				\leq & \ \frac{1}{2}\|x_0-y_0\|^2 + 2\|u-v\|_\infty(\operatorname{var}(x_u;[0,t]) + \operatorname{var}(x_v;[0,t]))\\
				& + \frac{2\mathscr{V}}{\rho}\|u-v\|_\infty^2 + \frac{1}{\rho}\int_{0}^t\|x_u(s)-x_v(s)\|^2 d\mu\\
				\leq & \ \frac{1}{2}\|x_0-y_0\|^2 +  4\mathscr{V}\|u-v\|_\infty + \frac{2\mathscr{V}}{\rho}\|u-v\|_\infty^2\\
				& + \frac{1}{\rho}\int_{0}^t\|x_u(s)-x_v(s)\|^2 d\mu.
			\end{aligned}
		\end{equation*}
From Gronwall's inequality (see Lemma \ref{gronwall-lemma}), for all $t\in [0,T]$
\begin{equation*}
\|x_u(t)-x_v(t)\|^2\leq \exp\left(\frac{4\mathscr{V}}{\rho}\right)\left(\|x_0-y_0\|^2 + 8\mathscr{V}\|u-v\|_\infty + \frac{4\mathscr{V}}{\rho}\|u-v\|_\infty^2\right).
\end{equation*}
Hence $\mathscr{C} := \exp\left(\frac{4\mathscr{V}}{\rho}\right)\max\{1,8\mathscr{V},4\mathscr{V}/\rho\}$ satisfies the claim. \\
(b) The set $\mathscr{B} := \{u_n\}_{n\in\N}\cup\{u\}$ is compact in $\mathcal{C}([0,T];\H)$ for the uniform topology; hence $\mathscr{B}$ is equicontinuous (see \cite[Theorem 47.1]{MR3728284}). Thus, using (a), we obtain $x_{u_n}(\cdot;x_0^n)\to x_u(\cdot;x_0)$ uniformly.
	\end{proof}

	\section{Stochastic Perturbation of the Sweeping Process}\label{sec7}
	
	In this section, $\H$ stands for a $d$-dimensional Euclidean space, i.e., $\H = \R^d$. Let $(\Omega, \mathcal{F},\mathbb{P})$ be a probability space and let $B$ be an $\ell$-dimensional Brownian motion. 
    Formally, we are interested in the well-posedness of the following stochastic sweeping process  
	\begin{equation}\label{spe-gen}
		\left\{\begin{array}{l}
			dX_t \in -N^P(C(t);X_t)dt + \sigma(t,X_t)dB_t + f(t,X_t)dt,\\
			X_0 = \xi,
		\end{array}\right.
	\end{equation}
	where $\xi\colon \Omega\to C(0)$ represents a random initial condition.  
    A reasonable formulation is to look for a pair of processes $(X_t, K_t)_{t\in [0,T]}$ such that 
	\begin{equation}\label{SPEn}
		\left\{\begin{array}{l}
        X_t \in C(t), \forall t \in[0, T],\\
			X_t=K_t + \int_0^t \sigma(s,X_s)dB_s + \int_0^tf(s,X_s)ds, \forall t\in [0,T]\\
			\frac{dK_t}{|dK_t|} \in -N^P(C(t);X_t),\ \left|dK\right| \text {-a.e., }\\
			X_0 = \xi.
		\end{array}\right.
	\end{equation}
 The inclusion in the third equation of \eqref{SPEn} is often formulated in an alternative way: for all $t\in [0,T]$, 
\begin{equation*} 
\var(K;[0,t]) = \int_0^t \mathbbm{1}_{\{X(\cdot)\in \bd C(\cdot)\}}(s) |dK_s|\text{ and } K_t = \int_0^t \zeta_s |dK_s|,
\end{equation*}
 where $|dK_s|$-a.e. $\zeta_s\in N(C(s);X_s)$ (see, e.g., \cite[Definition 3.8]{MR2812587} and \cite[eq. 11]{MR745330}). It is readily verified that these formulations are equivalent.  Defining $\mathcal{R}(X)_t := \int_0^t \sigma(s,X_s)dB_s + \int_0^tf(s,X_s)ds$, \eqref{SPEn} can be rewritten as
    \begin{equation}\label{SPEnnnn}
		\left\{\begin{array}{l}
        X_t \in C(t), \forall t \in[0, T],\\
			X_t=K_t + \mathcal{R}(X)_t, \forall t\in [0,T]\\
			\frac{dK_t}{|dK_t|} \in -N^P(C(t)-\mathcal{R}(X)_t;K_t),\ \left|dK\right| \text {-a.e., }\\
			X_0 = \xi.
		\end{array}\right.
	\end{equation}
    For every $\eta\in C(0)$ and $h\in \mathcal{C}_0([0,T];\R^d)$, we define (when it exists) the process $Z^h_\eta$ by
	\begin{equation}\label{SP_Yn}
		\left\{\begin{array}{l}
			\frac{dZ^{h}_{\eta}}{|dZ^{h}_\eta|}(t) \in -N^P(C(t)-h(t);Z^{h}_{\eta}(t)), \ \left|dZ^h_\eta\right| \text {-a.e., }\\
			Z^h_{\eta}(0) = \eta.
		\end{array}\right.
	\end{equation}
	The differential inclusion above is commonly known as the Skorokhod problem with normal reflection on the boundary of $C$ (now dependent on time), where the solution is typically considered as $Z^h_\eta + h$. The usual formulation of Skorokhod problem is actually described in a different way (see, e.g., \cite{MR873889, MR529332, MR745330, MR2812587, Skorokhod1961}), but in our case the cone of reflection is the normal cone, hence \eqref{SP_Yn} is an equivalent form. By virtue of \eqref{SPEnnnn}, the well-posedness of \eqref{SPEn} depends directly on the existence of solutions to \eqref{SP_Yn}. Our goal is to study the existence and uniqueness of solutions to \eqref{SPEn}, assuming that the moving set $C$ takes closed and uniformly prox-regular values, and it only moves continuously with respect to the Hausdorff distance. In this scenario, the problem \eqref{SP_Yn} has a unique solution provided the satisfaction of \ref{H4mgh} (see Corollary \ref{mc-ch-cor}).
    
    We now introduce some notation. As usual, for a  random variable $Z$, $\sigma(Z)$ denotes the $\sigma$-algebra generated by $Z$. When $(X_t)_{t\geq 0}$ is a continuous process, we set $\mathcal{F}_t^X := \sigma(X_s:s\leq t)$. As usual, for two given $\sigma$-algebras $\mathcal{F}$ and $\mathcal{G}$, $\mathcal{F}\vee\mathcal{G}$ denotes the smallest $\sigma$-algebra containing $\mathcal{F}$ and $\mathcal{G}$ simultaneously. For two continuous semimartingales $(M^1_t)_{t\geq 0}$ and $(M_t^2)_{t\geq 0}$, we denote their quadratic covariation by $t\mapsto [M^1,M^2]_t$, and use the standard shorthand $[M^1]_t := [M^1,M^1]_t$. Given two continuous stochastic processes $(X_t)_{t\geq 0}$ and $(Y_t)_{t\geq 0}$, if they share the same law we write $X\stackrel{\mathscr{L}}{=}Y$. In what follows, all the filtrations are supposed to satisfy the usual conditions, i.e. they are completed and right-continuous.
	
	Concerning $f\colon [0,T]\times \R^d\to \R^d$ and $\sigma\colon [0,T]\times \R^d\to \R^{d\times \ell}$, assume they are Carathéodory mappings, i.e. for all $t\in[0,T]$, $f(t,\cdot),\sigma(t,\cdot)$ are continuous and for all $x\in \R^d$, $f(\cdot,x),\sigma(\cdot,x)$ are measurable. Moreover, we assume that there exist $\beta_f\in L^p([0,T];\R_+)$ with $p> 1$ and $\beta_\sigma\in L^q([0,T];\R_+)$ with $q>2$ such that 
    \begin{equation}\label{eqn_bounded_fsigma}
        \sup_{x\in \R^d}\|f(t,x)\|\leq \beta_f(t) \text{ and } \sup_{x\in \R^d}\|\sigma(t,x)\|\leq \beta_\sigma(t) \text{ for a.e. }t\in [0,T].
	\end{equation}
    
    As usual, $\sigma_{\cdot,j}$ denotes the $j$-th column of $\sigma$ and $\sigma_{i,\cdot}$ denotes its $i$-th row. 
As in the theory of stochastic differential equations, we distinguish different types of solutions. 

    \begin{definition}[Weak solution]
        A \emph{weak solution} of \eqref{spe-gen} is a triple $(B,X,K)$ defined on $(\Omega, \mathcal{F},(\mathcal{F}_t), \mathbb{P})$ such that:
        \begin{enumerate}[label=(\roman*)]
			\item $(\Omega,\mathcal{F},(\mathcal{F}_t),\mathbb{P})$ is a filtered probability space.
			\item $B$ is an $\ell$-dimensional $(\mathcal{F}_t)$-Brownian motion starting from $0$. 
			\item $X$ is a $(\mathcal{F}_t)$-adapted $\R^d$-valued process and $K$ is is a $(\mathcal{F}_t)$-adapted $\R^d$-valued finite variation process such that $(X,K)$ satisfy \eqref{SPEn}.
		\end{enumerate}
    \end{definition}

    \begin{definition}[Strong solution]
        Given a probability space $(\Omega, \mathcal{F}, \mathbb{P})$, an $\ell$ dimensional Brownian motion $B$ with its natural filtration $(\mathcal{F}^B_t)$, and a random initial condition $\xi\colon \Omega \to C(0)$ independent of $B$, a pair $(X,K)$ is called a \emph{strong solution} of \eqref{spe-gen} if $X$ and $K$ are $\R^d$-valued processes adapted to the filtration $(\mathcal{F}^B_t \vee \sigma(\xi))_{t\ge 0}$, satisfy \eqref{SPEn}, and such that $K$ is of finite variation.

    \end{definition}
	\begin{definition}[Pathwise uniqueness]
		\emph{Pathwise uniqueness} for \eqref{spe-gen} holds, if, for every filtered probability space $(\Omega,\mathcal{F},(\mathcal{F}_t),\mathbb{P})$ and every $(\mathcal{F}_t)$-Brownian motion $B$, any two weak solutions $X$ and $X'$ of the stochastic sweeping process with $\mathbb{P}(X_0 = X_0') = 1$ are indistinguishable, i.e.,  $\mathbb{P}(X_t= X_t' \textrm{ for all } t\in [0,T])=1.$
	\end{definition}

    We will show that, if \ref{H4mgh} holds, then \eqref{SPEn} admits a weak solution. Note that no assumptions on the boundedness or smoothness of $C$ are required.
	
	\begin{theorem}\label{weak-existence}
		Suppose that $C\colon [0,T]\tto \R^d$ satisfies \ref{H1mgh}, \ref{H2mgh} and \ref{H4mgh}. Let $\mu$ be an initial distribution on $C(0)$. Then there exists a weak solution to the stochastic sweeping process \eqref{SPEn} with initial distribution $\mu$.
	\end{theorem}
	\begin{proof} 
To begin, consider any probability space $(\Omega,\mathcal{F},\mathbb{P})$ rich enough to carry a random variable $x_0\colon \Omega\to C(0)$ with law $\mu$ and a Brownian motion $B$ on the same space, such that $x_0$ is independent of $B$.\\
First, regarding the hypotheses on $C$, Corollary \ref{mc-ch-cor} ensures that the operator $(\eta,h)\mapsto Z^h_\eta$ is well-defined on $C(0)\times \mathcal{C}_0([0,T]; \mathbb{R}^d)$ (see \eqref{SP_Yn}). By continuity of this operator (Proposition \ref{lemma-cont-sol}) and uniqueness of solutions, for every random variable $\eta\colon \Omega\to C(0)$ and every continuous process $U$ with $U(0) = 0$ a.e.,  the process $Z^U_\eta$ is continuous and adapted to the filtration $(\mathcal{F}^U_t\vee\sigma(\eta))$.\\
Since $f$ and $\sigma$ are Carathéodory mappings, from \cite[Theorem 2.3]{Kubinska2004} there are sequences of continuous functions $(\hat{f}_n\colon [0,T]\times \R^d\to \R^d)$ and $(\hat{\sigma}_n\colon [0,T]\times \R^d\to \R^{d\times\ell})$ such that for a.e. in $[0,T]$, there is $n_t\in\N$ such that $\hat{f}_n(t,\cdot) = f(t,\cdot)$ and $\hat{\sigma}_n(t,\cdot) = \sigma(t,\cdot)$ for all $n>n_t$. We can also approximate $\beta_f, \beta_\sigma$ by a sequence of continuous functions $(\beta_f^n)\subset L^p([0,T];\R_+)$ and $(\beta_\sigma^n)\subset L^q([0,T];\R_+)$ such that $\beta_f^n\to \beta_f$ and $\beta_\sigma^n\to \beta_\sigma$ a.e. on $[0,T]$ and there are $\hat\beta_f\in L^p([0,T];\R_+)$ and $\hat\beta_\sigma\in L^q([0,T];\R_+)$ such that $|\beta_f^n|\leq \hat{\beta}_f$ and $|\beta_\sigma^n|\leq \hat\beta_\sigma$ a.e. on $[0,T]$, thus we can take $f_n := \min\left\{1,\frac{\beta_f^n}{\|\hat f_n\|}\right\}\hat f_n$ and  $\sigma_n := \min\left\{1,\frac{\beta_\sigma^n}{\|\hat\sigma_n\|}\right\}\hat{\sigma}_n$, and we have $f_n,\sigma_n$ are continuous on $[0,T]\times \R^d$ and \eqref{eqn_bounded_fsigma} holds with $\hat{\beta}_f, \hat\beta_\sigma$ (instead of $\beta_f, \beta_\sigma$), not depending on $n\in \N$.

For $n\in\N$, consider the partition of $[0,T]$ given by $t_i^n = \frac{i}{n}T$, $i\in\{0,1,\ldots,n\}$. Set $Y_n(t) := \int_{t_0^n}^tf_n(s,x_0)ds + \int_{t_0^n}^t\sigma_n(s,x_0)dB_s$ on $[t_0^n,t_1^n]$, and define $X_n(t) :=Y_n(t) + Z^{Y_n}_{x_0}(t)$ on $[t_0^n,t_1^n]$. Extend $Y_n$ to $[t_0^n,t_2^n]$ by 
		\begin{equation*}
			Y_n(t) = Y_n(t_1^n) + \int_{t_1^n}^{t}f_n(s,X_n(t_1^n))ds + \int_{t_1^n}^{t}\sigma_n(s,X_n(t_1^n))dB_s
		\end{equation*}
		for $t\in ]t_1^n,t_2^n]$, and set $X_n = Y_n + Z^{Y_n}_{x_0}$ on $[t_0^n,t_2^n]$. Proceeding inductively,  for $]t_2^n,t_3^n]$  define
		\begin{equation*}
			Y_n(t) = Y_n(t_2^n) + \int_{t_2^n}^{t}f_n(s,X_n(t_2^n))ds + \int_{t_2^n}^{t}\sigma_n(s,X_n(t_2^n))dB_s,
		\end{equation*}
        and so on. Thus we construct $(X_n,Y_n)$ with $X_n = Y_n + Z^{Y_n}_{x_0}$ on $[0,T]$ and  
		\begin{equation*}
			Y_n(t) = Y_n(t_k^n) + \int_{t_k^n}^{t}f_n(s,X_n(t_k^n))ds + \int_{t_k^n}^{t}\sigma_n(s,X_n(t_k^n))dB_s,  t\in [t_k^n,t_{k+1}^n].
		\end{equation*}
		Equivalently, for $t\in [0,T]$,
		\begin{equation*}
			Y_n(t) = \int_{0}^t f_n(s,X_n(\delta_n(s)))ds + \int_{0}^t \sigma_n(s,X_n(\delta_n(s)))dB_s,
		\end{equation*}
		where $\delta_n(t) = t_k^n$ if $t\in [t_k^n,t_{k+1}^n[$ for some $k<n$ and $\delta_n(t) = t_n^n$ if $t=t_n^n$.\\
\emph{\textbf{Claim 1}.} There exist a subsequence $(n_k)_{k\in\N}$, a probability space $(\tilde{\Omega},\tilde{\mathcal{F}},\tilde{\mathbb{P}})$,  processes $\{(\tilde{x}_{n_k},\tilde{B}_{n_k},\tilde{Y}_{n_k})\}_{k}$ and $(\tilde{x}_0,\tilde{B},\tilde{Y})$ in $C(0)\times \mathcal{C}_0([0,T];\R^d)\times \mathcal{C}_0([0,T];\R^d)$ such that
			\begin{enumerate}
\item [(a)] $(\tilde{x}_{n_k},\tilde{B}_{n_k},\tilde{Y}_{n_k}) \stackrel{\mathscr{L}}{=} (x_0,B,Y_{n_k})$ for all $k\in\N$.
\item [(b)] $(\tilde{x}_{n_k},\tilde{B}_{n_k},\tilde{Y}_{n_k}) \to (\tilde{x}_0,\tilde{B},\tilde{Y})$ $\Tilde{\mathbb{P}}$-a.e. on $C(0)\times \mathcal{C}([0,T];\R^\ell)\times \mathcal{C}([0,T];\R^d)$.
		\end{enumerate}
\noindent \emph{Proof of the Claim 1.} By \cite[Theorem 1.45]{MR3308895}, it suffices to show that $(x_0,B,Y_n)$ is tight. Clearly $(x_0,B)$ is tight. We check tightness of $(Y_n)$. By \cite[Theorem I.4.2 \& I.4.3]{zbMATH03780265} (see also \cite[Problem 4.11, p. 64]{MR1121940}), it is enough to prove 
		\begin{enumerate}
			\item [(i)] There exist $\gamma>0$ and $C_1>0$ such that $\mathbb{E}(\|Y_n(0)\|^\gamma)\leq C_1$ for all $n\in\N$.
			\item [(ii)] There exist $\alpha>0$, $\beta>0$, and $C_2>0$ such that for all $t,s\in [0,T]$ and $n\in\N$, 
        $$\mathbb{E}(\|Y_n(t)-Y_n(s)\|^\alpha)\leq C_2|t-s|^{1+\beta}.$$
		\end{enumerate}
		The first assertion is immediate since $Y_n(0) = 0$. To check (ii), take $$\kappa:= \min\left\{1-\frac{1}{p},\frac{1}{2}-\frac{1}{q}\right\}$$ and choose any $\alpha>0$ such that $\alpha\kappa>1$ and $\beta := \alpha\kappa-1$. We observe that
		\begin{equation*}
			\|Y_n(t)-Y_n(s)\|^\alpha \leq d^{\alpha/2 - 1}\sum_{i=1}^d |Y_n(t)_i-Y_n(s)_i|^\alpha.
		\end{equation*}
For $t,s\geq t_0^n$, by Burkholder-Davis-Gundy inequalities (see, e.g., \cite[Problem 3.29, p. 166]{MR1121940}) and Hölder inequality, for all $i,j$,
		\begin{equation*}
			\begin{aligned}
				&\mathbb{E}\left( \left|\int_{s}^{t}\sigma_n(\tau,X_n(\delta_n(\tau)))_{ij}dB_\tau^j\right|^\alpha \right)\\
				&\quad\quad   \quad\quad\quad\quad\quad\quad\quad\quad\quad\quad\quad \leq   C\cdot \mathbb{E}\left(\left(\int_{s}^{t}\sigma_n(\tau,X_n(\delta_n(\tau)))^2_{ij} d\tau\right)^{\alpha/2}\right)\\
				&\quad\quad \quad\quad\quad\quad\quad\quad \quad\quad\quad\quad\quad \leq   C\left(\int_s^t\hat{\beta}_\sigma(\tau)^2d\tau\right)^{\alpha/2}\\
                &\quad\quad \quad\quad\quad\quad\quad\quad \quad\quad\quad\quad\quad \leq   C\left(\int_s^t\hat{\beta}_\sigma(\tau)^qd\tau\right)^{\alpha/q}|t-s|^{\alpha\left(\frac{1}{2}-\frac{1}{q}\right)},
			\end{aligned}
		\end{equation*}
		for some $C>0$ (only dependent of $\alpha$), and we also have 
        \begin{equation*}
            \begin{aligned}
            \left(\int_{s}^{t} |f_n(\tau,X_n(\delta_n(\tau)))_i|d\tau\right)^\alpha &\leq \left(\int_s^t \hat\beta_f(\tau)d\tau\right)^\alpha \\
            &\leq \left(\int_s^t \hat\beta_f(\tau)^pd\tau\right)^{\alpha/p}|t-s|^{\alpha\left(1-\frac{1}{p}\right)}.
            \end{aligned}
        \end{equation*}
        Since $\beta_f\in L^p([0,T];\R_+)$ and  $\beta_\sigma\in L^q([0,T];\R_+)$, there exists a constant $\mathcal{K}>0$ (independent of $t,s,n$) such that 
		\begin{equation*}
				\mathbb{E}(\|Y_n(t) - Y_n(s)\|^\alpha)
				\leq  \mathcal{K}|t-s|^{\alpha\kappa} = \mathcal{K}|t-s|^{\beta+1}.
		\end{equation*}
		Thus $(Y_n)$ is tight, and therefore $(x_0,B,Y_n)$ is tight.\\ 
		Since $(x_0,B,Y_{n_k})\stackrel{\mathscr{L}}{=} (\tilde{x}_{n_k},\tilde B_{n_k},\tilde{Y}_{n_k})$, the continuity of $(x,\eta)\mapsto Z_{x}^\eta$ and the measurability of stochastic integral operator (see, e.g., \cite[Theorem 3]{WOS:A1995QX25200002}) yield that $\tilde{\mathbb{P}}$-a.e., for all $t\in [0,T]$,
        \begin{equation}\label{eqn-Ynk-bfr-limit}
            \tilde{Y}_{n_k}(t) = \int_{0}^t f_{n_k}(s,\tilde{X}_{n_k}(s))ds + \int_{0}^t \sigma_{n_k}(s,\tilde{X}_{n_k}(s))d\tilde{B}_{n_k}(s),
        \end{equation}
		where $\tilde{X}_{n_k} := Z^{\tilde{Y}_{n_k}}_{\tilde{x}_{n_k}} + \tilde{Y}_{n_k}$.\\
        \emph{\textbf{Claim 2}.} $\forall k\in \N$, $(X_{n_k},B) \stackrel{\mathscr{L}}{=} (\tilde{X}_{n_k},\tilde{B}_{n_k})$, and $B$ is a $(\mathcal{F}^{X_{n_k},B}_t)$-Brownian motion.\\ 
	\emph{Proof of Claim 2.}
        By construction, for all $n\in \N$, $X_n$ is adapted to $(\mathcal{F}^{B}_t\vee\sigma(x_0))$. Since $B$ and $x_0$ are independent, $B$ is an $(\mathcal{F}_t^B\vee \sigma(x_0))$-Brownian motion; hence,  for every $k\in \N$, $B$ is a $(\mathcal{F}^{X_{n_k},B}_t)$-Brownian motion. Moreover, $(x_0,Y_{n_k},B) \stackrel{\mathscr{L}}{=} (\tilde{x}_{n_k},\tilde{Y}_{n_k},\tilde{B}_{n_k})$ and $(x,\eta)\mapsto \eta + Z^{\eta}_x$ is continuous on $C(0)\times \mathcal{C}([0,T];\R^d)$ (see Proposition \ref{lemma-cont-sol}), so  $(X_{n_k},B) \stackrel{\mathscr{L}}{=} (\tilde{X}_{n_k},\tilde{B}_{n_k})$.\\
		\emph{\textbf{Claim 3}.} With $(\tilde{x}_0,\tilde{B},\tilde{Y})$ as above, the pair $(Z^{\tilde{Y}}_{\tilde{x}_0} + \tilde{Y}, Z^{\tilde{Y}}_{\tilde{x}_0})$ is a (weak) solution of \eqref{SPEn} on $(\tilde{\Omega},\tilde{\mathcal{F}},\tilde{\mathbb{P}})$. Let us define the process $\tilde{X}$ by $\tilde{X} := Z^{\tilde{Y}}_{\tilde{x}_0} + \tilde{Y}$.\\ 
\emph{Proof of the Claim 3.} Since $\tilde{Y}_{n_k}\to \tilde{Y}$ uniformly $\tilde{\mathbb{P}}$-a.e., Proposition \ref{lemma-cont-sol} yields  $Z^{\tilde{Y}_{n_k}}_{\tilde{x}_{n_k}}\to Z^{\tilde{Y}}_{\tilde{x}_0}$ uniformly $\tilde{\mathbb{P}}$-a.e., hence $\tilde{X}_{n_k}\to \tilde{X}$ uniformly $\tilde{\mathbb{P}}$-a.e. We now prove that,  $\tilde{\mathbb{P}}$-a.e., for all $t\in [0,T]$,
		\begin{equation}\label{eqn_lim}
			\tilde{Y}(t) = \int_{0}^t f(s,\tilde{X}(s))ds + \int_{0}^t \sigma(s,\tilde{X}(s))d\tilde{B}(s).
		\end{equation}
On the one hand, by uniform convergence $\tilde{X}_{n_k}\to \tilde{X}$, we have $$f_{n_k}(s,\tilde{X}_{n_k}(\delta_{n_k}(s)))\to f(s,\tilde{X}(s))$$ a.e. on $[0,T]$ as $k\to+\infty$. Since $(f_{n_k})$ are uniformly upper bounded by $\hat\beta_f$, the dominated convergence theorem implies that, $\tilde{\mathbb{P}}$-a.e., for all $t\in [0,T]$, 
		\begin{equation}\label{conv-f-weak}
			\lim_{k\to \infty}\int_{0}^t f_{n_k}(s,\tilde{X}_{n_k}(\delta_{n_k}(s)))ds =  \int_0^t f(s,\tilde{X}(s))ds.
		\end{equation}
        On the other hand, by virtue of Claim 3 and \cite[Corollary 1.96]{MR3308895}, for each $k$, $\Tilde{B}_{n_k}$ is a $(\mathcal{F}^{\tilde{B}_{n_k},\tilde{X}_{n_k}}_t)$-Brownian motion and $\tilde{B}$ is a $(\mathcal{F}^{\tilde{B},\tilde{X}}_t)$-Brownian motion. Hence, by \cite[Proposition 2.15]{MR3308895}, 
		\begin{equation*}
			\int_{0}^{t}\sigma_{n_k}(s,\tilde{X}_{n_k}(\delta_{n_k}(s)))d\tilde{B}_{n_k}(s)\to \int_{0}^{t}\sigma(s,\tilde{X}(s))d\tilde{B}(s)
		\end{equation*}
		uniformly in probability on $[0,T]$. Passing to a subsequence (see, e.g., \cite[Theorem 9.2.1]{MR1932358}), we obtain,  $\tilde{\mathbb{P}}$-a.e., for all $t\in [0,T]$,
        \begin{equation}\label{conv-sigma-weak}
            \lim_{i\to \infty}\int_{0}^{t}\sigma_{n_{k_i}}(s,\tilde{X}_{n_{k_i}}(\delta_{n_{k_i}}(s)))d\tilde{B}_{n_{k_i}}(s)= \int_{0}^{t}\sigma(s,\tilde{X}(s))d\tilde{B}(s).
        \end{equation}
         Then, taking into account the convergences \eqref{conv-f-weak} and \eqref{conv-sigma-weak}, we pass to the limit in \eqref{eqn-Ynk-bfr-limit} along the subsequence $(n_{k_i})$ and conclude that \eqref{eqn_lim} holds for all $t\in [0,T]$,  $\tilde{\mathbb{P}}$-a.e.
	\end{proof}
    

	Now, we establish pathwise uniqueness to \eqref{SPEn} by imposing that $f$ is one-sided Lipschitz and $\sigma$ is Lipschitz; that is, there exist $$\mathscr{L}_f\in L^1([0,T];\R_+) \text{ and }\mathscr{L}_\sigma\in L^2([0,T];\R_+)$$ such that, for all $t\in [0,T]$ and all $x,y\in \R^d$, 
	\begin{equation}\label{L-lips-fsig1}
		\langle x-y,f(t,x)-f(t,y)\rangle\leq \mathscr{L}_f(t)\|x-y\|^2.
	\end{equation}
and 
    \begin{equation}\label{L-lips-fsig2}
		\|\sigma(t,x)-\sigma(t,y)\|\leq \mathscr{L}_\sigma(t)\|x-y\|.
	\end{equation}
	
	\begin{theorem}\label{pathwise-uniqueness}
		 Assume that \ref{H1mgh} holds. Consider two weak solutions $(X^1,K^1)$, $(X^2,K^2)$ of \eqref{SPEn}, defined on the same filtered probability space $(\Omega,\mathcal{F},(\mathcal{F}_t),\mathbb{P})$ with the same Brownian motion $B$, such that $\mathbb{P}(X_0^1 = X_0^2) = 1$. Suppose that $f,\sigma$ satisfy \eqref{L-lips-fsig1} and \eqref{L-lips-fsig2}. Then 
        $$
        \mathbb{P}(X^1_t = X^2_t, \textrm{ for all } t\in [0,T]) = 1.
        $$    
	\end{theorem}
	\begin{proof}
Consider the stopping time 
$$
\tau_n := \inf\{ t\in [0,T] : \max\{\|X_t^1\|,\|X_t^2\|\}\geq n \}.
$$ 
Define $\xi := X_0^1 = X_0^2$ and $Y^i := Z^{X^i}_\xi$ for $i\in \{1,2\}$.  On $\R\times\R^d$, define $\phi(x,y):=\exp(-\frac{1}{\rho}x)\|y\|^2$, which is of class $\mathcal{C}^2(\R\times \R^d)$, and the processes $U_t := \operatorname{var}(K^1;[0,t]) + \operatorname{var}(K^2;[0,t])$ and $V_t := X_t^1-X_t^2$. By Itô's formula (see, e.g., \cite[Theorem 3.6, p. 153]{MR1121940}),
\begin{equation}\label{ito-uniqueness}
\begin{aligned}
\phi(U_t,V_t^{\tau_n}) = & \ \frac{-1}{\rho}\int_{0}^t \exp(-U_s/\rho)\|V_s^{\tau_n}\|^2 dU_s +\int_{0}^t 2\exp(-U_s/\rho)\langle V_s^{\tau_n},dV_s^{\tau_n}\rangle\\
&+\sum_{i=1}^d \int_{0}^t \exp(-U_s/\rho) d[(V^{\tau_n})^i]_s,
\end{aligned}
\end{equation}
where $V^{\tau_n}_t := V_{\tau_n\wedge t}$, and we have used that $U$ has bounded variation; hence its quadratic variation is zero. Note that 
		\begin{equation*}
\begin{aligned}
				\int_{0}^t 2\exp(-U_s/\rho)&\langle V_s^{\tau_n},dV_s^{\tau_n}\rangle 
				 =   \ \int_0^{t\wedge\tau_n} 2\exp(-U_s/\rho)\langle V_s^{\tau_n},dK_s^1-dK_s^2\rangle\\
				&\quad  + \int_0^{t\wedge\tau_n} 2\exp(-U_s/\rho) \langle V_s^{\tau_n}, f(s,X_s^1)-f(s,X_s^2)\rangle ds\\
				&\quad + \int_0^{t\wedge\tau_n}  2\exp(-U_s/\rho) \langle (\sigma(s,X_s^1)-\sigma(s,X_s^2))^\top V_s^{\tau_n},dB_s\rangle,
			\end{aligned}
		\end{equation*}
		and 
		\begin{equation*}
			\sum_{i=1}^d \int_{0}^t \exp(-U_s/\rho) d[(V^{\tau_n})^i]_s = \int_0^{t\wedge\tau_n} \exp(-U_s/\rho)\|\sigma(s,X_s^1) - \sigma(s,X_s^2)\|^2 ds.
		\end{equation*}
		Moreover, due to the uniform prox-regularity of the moving sets $C(t)$, we have 
		\begin{equation*}
			\int_0^{t\wedge\tau_n} 2\exp(-U_s/\rho)\langle V_s^{\tau_n},dK_s^1-dK_s^2\rangle\leq \frac{1}{\rho}\int_{0}^t \exp(-U_s/\rho)\|V_s^{\tau_n}\|^2 dU_s.
		\end{equation*}
Using \eqref{L-lips-fsig1} and \eqref{L-lips-fsig2} together with the preceding equations in \eqref{ito-uniqueness}, we obtain, for all $t\in [0,T]$,
\begin{equation*}
			\begin{aligned}
				\phi(U_t,V_t^{\tau_n})\leq & \ \int_0^{t\wedge\tau_n} (\mathscr{L}_\sigma(s)^2+2\mathscr{L}_f(s))\phi(U_s,V_s^{\tau_n})ds\\
				&\quad\quad + \int_0^{t\wedge\tau_n}  2\exp(-U_s/\rho) \langle (\sigma(s,X_s^1)-\sigma(s,X_s^2))^\top V_s^{\tau_n},dB_s\rangle.
			\end{aligned}
		\end{equation*}
        Given that the second term of the right hand side is a martingale, after taking expectation we obtain that 
        \begin{equation}\label{eqn-expect-before-gronwall}
            \begin{aligned}
            \mathbb{E}(\phi(U_t,V_t^{\tau_n}))&\leq   \mathbb{E}\left(\int_0^{t\wedge\tau_n} (\mathscr{L}_\sigma(s)^2+2\mathscr{L}_f(s))\phi(U_s,V_s^{\tau_n})ds\right)\\
            &\leq \int_0^t (\mathscr{L}_\sigma(s)^2+2\mathscr{L}_f(s))\mathbb{E}(\phi(U_s,V_s^{\tau_n}))ds
            \end{aligned}
        \end{equation}
Since $\mathscr{L}_\sigma\in L^2([0,T];\R_+)$ and $\mathscr{L}_f\in L^1([0,T];\R_+)$ we have $s\mapsto \mathscr{L}_\sigma(s)^2+2\mathscr{L}_f(s)$ is integrable, and by definition of $\tau_n$ we have $s\mapsto \mathbb{E}(\phi(U_s,V_s^{\tau_n}))$ is continuous. Gronwall's lemma then implies 
\begin{equation*}
    \mathbb{E}\left( \phi(U_s,V_s^{\tau_n})\right) = 0, \forall s\in [0,T], \forall n\in \N.
\end{equation*}
For $s\in [0,T]$ fixed, we have for all $n\in \N$, $\mathbb{P}$-a.e. $\phi(U_s,V_s^{\tau_n})= 0$, and therefore, for all $n\in \mathbb{N}$,
\begin{equation*}
		      \mathbb{P}\text{-a.e. } \|X_{s\wedge\tau_n}^1-X_{s\wedge\tau_n}^2\|^2 = 0.
		  \end{equation*}
          Since $\mathbb{P}$-a.e. $\tau_n\nearrow T$, we have that $\mathbb{P}$-a.e. $X_s^1 = X_s^2$. By using \cite[Problem 1.1.5]{MR1121940} we conclude that $\mathbb{P}(X^1_t = X^2_t, \forall t\in [0,T]) = 1$.
	\end{proof}
\begin{remark}
    It is worth noting that, in the stationary case, a related result can be found in \cite[Lemma 5.5]{MR873889}, where pathwise uniqueness is established in a constrained sense (up to a certain exit time) and under specific geometric assumptions. In contrast, our approach provides pathwise uniqueness without requiring such geometric conditions (namely, \ref{H4mgh}).
\end{remark}

Now, our goal is to pass from weak existence to strong existence by using a variant of the Yamada-Watanabe theorem (see \cite{MR278420}), as used in \cite{MR1392450}. Our result extends \cite[Theorem 3.11]{MR2812587} because we consider continuous moving sets (rather than absolutely continuous ones),  and  \ref{H4mgh} is more general (and more manageable) than the admissibility condition considered there.
	
\begin{theorem}\label{thm-strong-existence}
Consider a set-valued mapping $C\colon [0,T]\to \H$ satisfying \ref{H1mgh}, \ref{H2mgh} and \ref{H4mgh}, and that $f,\sigma$ satisfy \eqref{L-lips-fsig1} and \eqref{L-lips-fsig2}. Then, for every choice of probability space $(\Omega,\mathcal{F},\mathbb{P})$, Brownian motion $B$, and random initial condition $x_0\colon \Omega\to C(0)$, independent of $B$, there exists a pathwise-unique strong solution to the stochastic sweeping process \eqref{SPEn}.
	\end{theorem}
	\begin{proof}
Consider the sequence $(Y_n)$ generated in the proof of Theorem \ref{weak-existence} from $(x_0,B)$. Let $(n_k)$ and $(m_k)$ be two subsequences. We will prove that there is a further subsequence such that $(Y_{n_{k_\ell}},Y_{m_{k_\ell}})$ converges in law to some process $(Z,Z)$,  where $Z$ is a random variable, defined on some probability space, taking values on $\mathcal{C}([0,T];\R^d)$.\\
Consider the sequence $(x_0,B,Y_{n_{k}},Y_{m_{k}})_k$. The same argument used in the proof of Theorem \ref{weak-existence} shows that the sequence is tight;  therefore, there exist processes  $(\tilde{x}_{n_{k_\ell}})$, $(\tilde{B}_{n_{k_\ell}})$, $(\tilde{Y}_{n_{k_\ell}})$, $(\tilde{Y}_{m_{k_\ell}})$, all defined on the same stochastic basis, such that they converge a.e. to some random variables $\tilde{x}_0$, $\tilde{B}$, $\tilde{Y}^1$, $\tilde{Y}^{2}$, respectively, and for all $\ell$ 
 	\begin{equation}\label{eqn-eq-law}
 		(x_0,B,Y_{n_{k_\ell}},Y_{m_{k_\ell}}) \stackrel{\mathscr{L}}{=} (\tilde{x}_{n_{k_\ell}},\tilde{B}_{n_{k_\ell}},\tilde{Y}_{n_{k_\ell}},\tilde{Y}_{m_{k_\ell}}).
 	\end{equation}
 Define $\tilde{X}^j := Z^{\tilde{Y}^{j}}_{\tilde{x}_0} + \tilde{Y}^j$ for $j\in \{1,2\}$. By the same argument as in Theorem \ref{weak-existence}, $(\tilde{X}^j,Z^{\tilde{Y}^{j}}_{\tilde{x}_0})$ is a solution of \eqref{SPEn} starting from $\tilde{x}_0$, driven by the same Brownian motion $\tilde{B}$, which turns out to be a $(\mathcal{F}_t^{\tilde{B},\tilde{X}^1,\tilde{X}^2})$-Brownian motion. Then  pathwise uniqueness (Theorem \ref{pathwise-uniqueness}) implies that  $\tilde{X}^1 = \tilde{X}^2$ a.e., hence  $\tilde{Y}^1 = \tilde{Y}^2$ a.e. Choose $Z = \tilde{Y}^1$. Since $(\tilde{Y}_{n_{k_\ell}},\tilde{Y}_{m_{k_\ell}})$ converges a.e. to $(Z,Z)$, by \eqref{eqn-eq-law} we conclude that $(Y_{n_{k_\ell}}, Y_{m_{k_\ell}})$ converges in law to $(Z,Z)$. By \cite[Lemma 1.1]{MR1392450}, $(Y_n)$ converges in probability to some random variable $Y\colon \Omega\to \mathcal{C}([0,T];\R^d)$; thus we can extract a subsequence $(Y_{n_i})$ converging a.e. to $Y$ (see, e.g., \cite[Theorem 9.2.1]{MR1932358}). Using again the argument from Theorem \ref{weak-existence}, we obtain that $(X,K)$ is a solution of \eqref{SPEn}, with $X := Z^Y_{x_0} + Y$ and $K:=Z^Y_{x_0}$; it is in fact strong, since,  by construction, for all $n\in \N$, $X_n$ is $(\mathcal{F}^B_t\vee \sigma(x_0))$-adapted. This yields the result.
	\end{proof}
\begin{remark}
    Under condition \ref{H4mgh}, the stationary case of $C$ was already addressed in \cite[Theorem 41]{Buckdahn2015}. Indeed, it appears there as a particular instance of a more general framework: the authors consider a stochastic differential inclusion driven by the subdifferential of a semiconvex function, rather than by the normal cone of a uniformly prox-regular set.
\end{remark}
    
    \begin{remark}
        In \cite[Theorem 3.2 \& Corollary 3.3] {Costantini2006Boundary}, a similar result is established, where the moving set is bounded, time-dependent, and satisfies a smoothness assumption stronger than the requirement that both maps $t\tto C(t)$ and $t\tto\cl(\H\setminus C(t))$ are uniformly prox-regular. By Proposition \ref{rmk-pr-cpr}, in the finite-dimensional case this condition implies \ref{H4mgh}. Therefore, Theorem \ref{thm-strong-existence} extends the aforementioned result. On the other hand, although \cite{MR2683628} considers a moving set that varies continuously with respect to the Hausdorff distance and allows for oblique reflection, normal reflection does not satisfy the assumptions in the general case, since Hausdorff-continuity of the reflection cone is required.
    \end{remark}
    
    \section*{Acknowledgements}

    \begin{itemize}
        \item J. G. Garrido was supported by ANID Chile under grants CMM BASAL funds for Center of Excellence FB210005, Project ECOS230027, MATH-AMSUD 23-MATH-17, and ANID BECAS/DOCTORADO NACIONAL 21230802.
        \item N. Kazi-Tani was supported by the ECOS–ANID project C23E06 and the French Research Agency through the ANR project DREAMeS (ANR-21-CE46-0004).
        \item E. Vilches was supported by ANID (Chile) through Fondecyt Regular grants No.~1220886, No.~1240120, and No.~1261728, CMM BASAL funds for the Center of Excellence FB210005, project ECOS230027, MATH-AmSud 23-MATH-17.
    \end{itemize}

    {\small
	\bibliographystyle{elsarticle-harv}
    \bibliography{references}
    }
 \appendix
 \begingroup
 \small
 \setlength{\baselineskip}{0.9\baselineskip}
 
 \section{Technical Proofs}
\subsection{Proof of Lemma \ref{C(t)H-statement}}\label{C(t)=H.proof}
\noindent Consider $t\in [0,T]$. The Hausdorff-continuity allows to get a finite sequence $t_1,t_2,...,t_k$ such that $d_H(C(t_i),C(t_{i+1}))<\rho$ for all $i=0,1,...,k$ where $t_{k+1} := t$. Since $C(t_0) = \H$, we are going to prove that $C(t_1) = \H$. Indeed, since $d_H(C(t_0),C(t_1))<\rho$, we have for all $x\in \H$, $d(x;C(t_1))<\rho$. Suppose $C(t_1)\neq \H$, i.e. there exists $x\in \H\setminus C(t_1)$. Since $C(t_1)$ is $\rho$-uniformly prox-regular and $d(x;C(t_1))<\rho$, $p:=\proj_{C(t_1)}(x)$ is well-defined and $p\neq x$. Moreover, $C(t_1)\cap \mathbb{B}_{\rho}(p+\rho\frac{x-p}{\|x-p\|}) = \emptyset$ (due to the uniform prox-regularity), then $d\left(p+\rho\frac{x-p}{\|x-p\|};C(t_1)\right)\geq \rho$ which is a contradiction, thus $C(t_1) = \H$. Next, we can proceed by an induction process in order to get that $C(t_i) = \H$ for all $i=1,...,k+1$, therefore $C(t) = \H$, concluding the desired.\qed

\subsection{Proof of Lemma \ref{lemma_partition}}\label{appendix1}

We have $\B_{r}(\ov x)\subset C(\ov t)$. Define $\beta:=\min\{r/2,\rho/2\} $, then $\delta= \mathscr{P}_{C}^{-1}(\beta)$. In particular,
\begin{equation}\label{eqn1-1}
d_H(C(t),C(\ov t))<\beta, \forall t\in ]\ov t-\delta, \ov t +\delta[\cap [0,T].
\end{equation}
Take $t\in ]\ov t-\delta, \ov t +\delta[\cap [0,T]$, we shall prove that $\B_{r/2}(\ov x)\subset C(t)$. Indeed, suppose this does not hold.  Then, there exists $a\in \B_{r/2}(\ov x)\setminus C(t)$, and \eqref{eqn1-1} yields $0<d(a;C(t)) <\beta$. Since $C(t)$ is $\rho$-uniformly prox-regular and $\beta<\rho$, we have that $\proj_{C(t)}(a)$ is well-defined. Consider  
$$y_\alpha:=\text{proj}_{C(t)}(a)+\frac{\alpha}{d(a;C(t))}(a-\text{proj}_{C(t)}(a)).$$ It is clear that for $0<\alpha<\rho$, $\proj_{C(t)}(y) = \proj_{C(t)}(a)$, hence $d(y;C(t)) = \alpha$. On the other hand,  
\begin{equation*}
    \|\ov x-y_\alpha\|\leq \|\ov x-a\| + \|a-y_\alpha\| < r/2+\|a-y_\alpha\|
     = r/2+ |d(a;C(t))-\alpha|.  
\end{equation*}
Taking $\alpha := d(a;C(t))+\beta$, we have $0<\alpha<2\beta\leq \rho$ and $$\|\ov x-y_\alpha\|< r/2+\beta\leq r/2 + r/2 = r,$$ thus $y_\alpha \in \B_{r}(\ov x)\subset C(\ov t)$, and by \eqref{eqn1-1} yields $d(y_\alpha;C(t))<\beta$, but it cannot happen since $d(y_\alpha;C(t)) = \alpha$ and $d(a;C(t))>0$. We get a contradiction, therefore $\mathbb{B}_{r/2}(\ov x)\subset C(t), \forall t\in ]\ov t-\delta, \ov t+\delta[\cap [0,T]$.\qed

\subsection{Proof of Proposition \ref{h1h2-bd}}\label{appendix2}

		Let $t\in [0,T]$ such that $\bd C(t)\neq \emptyset$ and $\epsilon>0$. Take $\epsilon' := \min\{\epsilon,\rho/2\}$. By the Hausdorff continuity of $C$, there exists $\delta>0$ such that $d_H(C(s),C(t))<\epsilon'$ for all $s\in ]t-\delta,t+\delta[$. Fix any such $s\in ]t-\delta,t+\delta[$ with $\bd C(s)\neq \emptyset$. Then there exists $\eta<\epsilon'$ such that 
        $$
        C(t)\subset C(s)+\eta\B \textrm{ and } C(s)\subset C(t) + \eta\B.
        $$
        First, let $x\in \bd C(t)\setminus \inte C(s)$. Since $d(x;C(s))\leq \eta<\rho$, the prox-regularity of $C(s)$ ensures that the metric projection $\proj_{C(s)}(x)$ is well-defined and belongs to $\bd C(s)$. Hence, 
        $$
        d(x;\bd C(s))\leq \|x-\proj_{C(s)}(x)\| = d(x;C(s)) \leq d(x;\bd C(s)),
        $$
        and thus $d(x;\bd C(s)) = d(x;C(s))$. Second, let $x\in \bd C(t)\cap\inte C(s)$ and suppose, by contradiction, that there exists $\gamma>0$ with $\mathbb{B}_{\eta+\gamma}(x)\cap \bd C(s) = \emptyset$. Then  $\mathbb{B}_{\eta+\gamma}(x)\subset \inte C(s)$, so $\mathbb{B}_{\eta+\gamma}(x)\subset C(t) + \eta\mathbb{B}$. By \cite[Corollary 1.6.2]{Clarke1998}, there exists $\bar x\in \bd C(t)\cap \mathbb{B}_{\gamma/2}(x)$ with $N^P(C(t);\bar x)\neq\{0\}$. Hence $\mathbb{B}_{\eta+\gamma/2}(\bar x)\subset C(t)+\eta\mathbb{B}$. Pick $v\in N^P(C(t);\bar x)\setminus\{0\}$ and choose  $\alpha\in ]\eta , \min\{\rho,\eta+\gamma/2\}[$. Then 
        $$
        \bar x + \frac{\alpha}{\|v\|}v\in \mathbb{B}_{\eta+\gamma/2}(\bar x)\subset C(t) + \eta\mathbb{B},
        $$ whereas, by proximal normality, $d(\bar x + \frac{\alpha}{\|v\|}v;C(t)) = \alpha>\eta$, which is a contradiction. Therefore, for every $\gamma>0$ we must have $\mathbb{B}_{\eta+\gamma}(x)\cap \bd C(s)\neq\emptyset$, and hence  $d(x;\bd C(s))\leq \eta$ for all $x\in \bd C(t)\cap\inte C(s)$. Combining the two cases,  
		\begin{equation*}
				\sup_{x\in \bd C(t)}d_{C(s)}(x) \leq \sup_{x\in \bd C(t)}\max\{d_{C(s)}(x),\eta\}\leq \max\bigl\{\eta,\sup_{x\in C(t)}d_{C(s)}(x)\bigr\}.   
		\end{equation*}
		By symmetry, the same argument yields
        $$\sup_{x\in \bd C(s)}d(x;\bd C(t)) \leq \max\bigl\{\eta,\sup_{x\in C(s)}d(x;C(t))\bigr\}.
        $$
        Therefore, $d_H(\bd C(s),\bd C(t))\leq \max\{\eta,d_H(C(s),C(t))\}<\epsilon'<\epsilon$. Since $\epsilon>0$ was arbitrary, it follows that $t\tto \bd C(t)$ is Hausdorff continuous on its domain.\qed

\subsection{Proof of Lemma \ref{upper-semi-support}}\label{appendix4}
 Take sequences $\xi_n\to \ov \xi$, $t_n\to \ov t$ and $x_n\to \ov x$ such that 
		\begin{equation*}
			\limsup_{\substack{x\to \ov x, t\to \ov t\\ \xi\to \ov \xi}} \sigma(\xi;\partial_P d_{C(t)}(x)\cap \mathbb{S}) = \lim_{n\to \infty} \sigma(\xi_n;\partial_P d_{C(t_n)}(x_n)\cap \mathbb{S}).
		\end{equation*}
		By compactness, for each $n\in\N$ there exists $v_n\in\partial_P d_{C(t_n)}(x_n)\cap\mathbb{S}$ such that $$\sigma(\xi_n;\partial_P d_{C(t_n)}(x_n)\cap \mathbb{S})=\langle\xi_n,v_n\rangle .$$ Passing to a subsequence if necessary, we may assume $v_n\to v\in \mathbb{S}$. We now show that $v\in \partial_P d_{C(\ov t)}(\ov x)$. Take any $y\in C(\ov t)$. Then there exists $N\in\N$ such that for all $n\geq N$, $y\in C(t_n)+\frac{\rho}{2}\mathbb{B}$. Since $x_n\to x$, we may also assume $x_n\in C(t_n) + \frac{\rho}{2}\mathbb{B}$ for all $n\geq N$. By \cite[Theorem 5]{MR4431268}, for all $n\geq N$ we have
		\begin{equation*}
			\begin{aligned}
				\langle v_n,y-x_n\rangle &\leq \frac{2}{\rho}\|y-x_n\|^2 + d_{C(t_n)}(y) - d_{C(t_n)}(x_n)\\
				&\leq \frac{2}{\rho}\|y-x_n\|^2 + d_H(C(t_n),C(\ov t)).
			\end{aligned}
		\end{equation*}
		Letting $n\to \infty$, we obtain $\langle v,y-\ov x\rangle\leq \frac{2}{\rho}\|y-\ov x\|^2$. Since this holds for all $y\in C(\ov t)$, it follows that $v\in N^P(C(\ov t);\ov x)$, and hence $v\in \partial_P d_{C(\ov t)}(\ov x)$. Finally, we observe
		\begin{equation*}
				\lim_{n\to \infty} \sigma(\xi_n;\partial_P d_{C(t_n)}(x_n)\cap \mathbb{S}) = \lim_{n\to\infty}\langle \xi_n,v_n\rangle = \langle\ov \xi,v\rangle \leq \sigma(\ov \xi;\partial_P d_{C(\ov t)}(\ov x)\cap \mathbb{S}),
		\end{equation*}
        which completes the proof.\qed

\subsection{Proof of Proposition \ref{prop-sublevel-h4-h5}}\label{appendix3}

Fix $t\in [0,T]$ and $x\in  C(t)$ with $I_\epsilon(t,x)\neq \emptyset$. By the classical min-max theorem (see, e.g., \cite[Lemma 1.85]{MR2986672}), we have
 \begin{equation}\label{min-max}
 	\begin{aligned}
 		&\inf_{v\in \mathbb{B}}\max_{i\in I_\epsilon(t,x)}\left\langle v,\nabla g_i(t,x)\right\rangle\\ 
 		=& -\min\left\{\left\|\sum_{i\in I_\epsilon(t,x)}\lambda_i \nabla g_i(t,x)\right\| : \sum_{i\in I_\epsilon(t,x)}\lambda_i = 1, \lambda_i\geq 0\right\}=:\xi(t,x).
 	\end{aligned}
 \end{equation}
 \emph{\textbf{Claim 1.}} Define $\kappa := \sup\{\xi(t,x) : t\in [0,T], x\in C(t), i\in I_\epsilon(t,x)\}$. Then $\kappa<0$.\\
 \emph{Proof of Claim 1.} Suppose not.  Then there exist sequences $(t_n)\subset [0,T]$ and $(x_n)$ with $x_n\in C(t_n)$ and $I_\epsilon(t_n,x_n)\neq \emptyset$ such that 
 \begin{equation*}
 	\left\|\sum_{i\in I_\epsilon(t_n,x_n)}\lambda_i^n\nabla g_i(t_n,x_n)\right\|\leq \frac{1}{n},
 \end{equation*}
 with $\sum_{i\in I_\epsilon(t_n,x_n)} \lambda_i^n = 1$ and $\lambda_i^n\geq 0$. For every $i\in \{1,...,m\}$, define $S_i := \{n\in \N : i\in I_\epsilon(t_n,x_n)\}$ and let $J = \{j : S_j \text{ is infinite}\}$. Clearly $J\neq \emptyset$. It cannot happen that $\lambda_i^n$ vanish along $S_i$, since the sum of coefficients is always one. Hence  there exists $\ell\in J$ and a subsequence $(n_k)\subset S_{\ell}$ such that $ \lambda_\ell^{n_k} \to \lambda_{\ell}^{\ast}>0$. Since $-\epsilon\leq g_\ell(t_{n_k},x_{n_k})\leq 0$, assumption $(a)$ and continuity imply that  $(t_{n_k},x_{n_k}) \to (\bar t,\bar x)$ with $\bar x\in C(\bar t)$ and $\ell \in I_\epsilon(\bar t,\bar x)$. Passing to the limit yields
 $$\lambda_\ell^\ast\nabla g_\ell(\bar t,\bar x) + \sum_{i\in\mathscr{I}} \lambda_i\nabla g_i(\bar t,\bar x) = 0 \quad \textrm{ for some } \mathscr{I}\subset I_\epsilon(\bar t,\bar x).$$
Positive linear independence then forces all coefficients to vanish, a contradiction.\\
\emph{\textbf{Claim 2:}} For all $t\in [0,T]$ and $x\in C(t)$ with $I_\epsilon(t,x)\neq \emptyset$, there is $v_{t,x}\in \mathbb{B}\setminus \{0\}$ such that 
 $$
 \langle v_{t,x},\nabla g_i(t,x)\rangle\leq \kappa \textrm{ for all } i\in I_\epsilon(t,x).
 $$
\noindent \emph{Proof of Claim 2.} Since $I_\epsilon(t,x)\neq \emptyset$, this follows directly from \eqref{min-max} and Claim 1.\\
\emph{\textbf{Claim 3:}} \ref{H4mgh} holds.\\ 
 \emph{Proof of Claim 3.} Let  
 $$
 \ell := \min\left\{1,\frac{\eta}{4},\frac{-\kappa}{4L},\frac{\epsilon}{2\mathsf{K}}\right\}>0,\, r := \min\left\{1,\frac{\eta}{8},\frac{-\kappa\ell}{4(2L+\mathsf{K})},\frac{\epsilon}{4\mathsf{K}}\right\},\, M := \frac{\ell}{r},
 $$
 where $\mathsf{K} := \sup\{\|\nabla g_i(t,x)\| : (t,x)\in \mathscr{Q}_\epsilon+\eta\mathbb{B},  i\in I\}$. We observe that, by virtue of assumptions $(a)$ and $(b)$,  the constant $\mathsf{K}$ is finite. \\
 Fix $t\in [0,T]$ and $x\in \bd C(t)$, and set $z := x+\ell v_{t,x}$. Then $\|z-x\|\leq Mr$. For any $\lambda\in [0,1]$ and $z'\in \mathbb{B}_{2r}(z)$, define $y := \lambda x + (1-\lambda)z'$. We show  $y\in C(t)$, which implies $\operatorname{co}(\{x\}\cup\mathbb{B}_{2r}(z))\subset C(t)$ and hence \ref{H4mgh}. First, 
 \begin{equation}\label{desigualdad-y}
 \|y-x\|= (1-\lambda)\|x-z'\|\leq \|x-z\| + 2r \leq \ell + 2r \leq \min\{\eta/2, \epsilon/\mathsf{K}\}.
 \end{equation}
 Thus $(t,y)\in \mathscr{Q}_\epsilon+\frac{\eta}{2} \mathbb{B}\subset \mathscr{Q}_\epsilon+\eta \mathbb{B}$, so $\Vert \nabla g_i(t,y)\Vert \leq \mathsf{K}$ for all $i\in I$.\\
 \noindent\emph{ Case 1: $i\notin I_\epsilon(t,x)$.} \\
 Then $g_i(t,x)<-\epsilon$, and by \eqref{desigualdad-y},
\begin{equation*}
	\begin{aligned}
		g_i(t,y)&\leq g_i(t,y) - g_i(t,x) + g_i(t,x)
        \leq \mathsf{K}\|y-x\| - \epsilon\leq \mathsf{K}(2r + \ell) - \epsilon \leq 0.
	\end{aligned}
\end{equation*}
\noindent\emph{Case 2: $i\in I_\epsilon(t,x)$.} \\
Write $y-x=(1-\lambda)(\ell v_{t,x}+\delta)$ where $\delta:=z'-z$ and $\Vert \delta\Vert \leq 2r$. By the descent lemma (since $\nabla g_i(t,\cdot)$ is $L$-Lipschitz on $C(t)+\eta\mathbb{B}$),
\begin{equation*}
\begin{aligned}
g_i(t,y)
&\leq g_i(t,x)+\langle \nabla g_i(t,x),\,y-x\rangle+\frac{L}{2}\|y-x\|^{2} \\
&\leq   (1-\lambda)\!\left(\ell\left\langle \nabla g_i(t,x), v_{t,x}\right\rangle
      + \left\langle \nabla g_i(t,x), \delta \right\rangle\right)
   + \frac{L}{2}(1-\lambda)^{2}\|\ell v_{t,x}+\delta\|^{2} \\
&\leq  (1-\lambda)\!\left(\ell \kappa + \|\nabla g_i(t,x)\|\,\|\delta\|\right)
   + L(1-\lambda)^{2}\!\left(\ell^{2} + \|\delta\|^{2}\right) \\
&\leq \ell \kappa + 2r\,\mathsf{K} + L\ell^{2} + 4L r^{2}.
\end{aligned}
\end{equation*}
Since $r\leq 1$, we have 
$$
g_i(t,y)\leq \ell \kappa +L\ell^2+r(2\mathsf{K}+4L)\leq \frac{\kappa \ell}{2}<0,
$$
where we have used $r\leq \frac{-\kappa \ell}{4(2L+\mathsf{K})}$ and $\ell\leq \frac{-\kappa}{4L}$ (recall $\kappa <0$). In both cases $g_i(t,y)\leq 0$, hence $y\in C(t)$. This proves \ref{H4mgh}.\\
  \noindent \emph{{\bf Claim 4.} $t\tto C(t)$ is Hausdorff-continuous.}\\
  \emph{Proof of Claim 4.} Consider $\beta\in ]0,\min\{1,\frac{\epsilon}{2\mathsf{K}},\epsilon,-\kappa/L,\eta,\eta'\}[$. By equicontinuity in $(b)$,  there exists $\vartheta>0$ such that for all $t',s'\in [0,T]$ with $|t'-s'|<\vartheta$, 
  $$
  \forall i\in I, \forall x'\in \mathsf{D}+\eta'\mathbb{B} : |g_i(t',x')-g_i(s',x')|\leq \min\{\epsilon/2,-\beta\kappa/2\}.$$
  where $\mathsf{D} = \bigcup_{s\in [0,T]} C(s)$. Fix $t,s\in [0,T]$ with $|t-s|<\vartheta$, and let $x\in C(t)$. On the one hand, if $I_\epsilon(t,x) = \emptyset$, then $g_i(t,x)<-\epsilon$ for all $i\in I$. Hence, 
  \begin{equation*}
  	g_i(s,x) = g_i(s,x)-g_i(t,x) + g_i(t,x)\leq \frac{\epsilon}{2} - \epsilon <0
  \end{equation*} 
so $x\in C(s)$. Hence $C(t)\subset C(s)\subset C(s)+\beta\mathbb{B}$. On the other hand, assume that $I_\epsilon(t,x)\neq \emptyset$. By Claim 2, pick $v\in \mathbb{B}\setminus \{0\}$ with $\langle v,\nabla g_i(t,x)\rangle\leq \kappa$ for all $i\in I_\epsilon(t,x)$. Note that $\beta<\eta'$. We show $x+\beta v\in C(s)$. \\
-If $i\notin I_\epsilon(t,x)$, then $g_i(t,x)\leq -\epsilon$. Moreover, since $\beta \leq \epsilon/(2\mathsf{K})$, one has
	\begin{equation*}
		\begin{aligned}
			g_i(s,x+\beta v)& = g_i(s,x+\beta v)- g_i(t,x+\beta v) +g_i(t,x+\beta v)-g_i(t,x) + g_i(t,x)\\
			&\leq \frac{\epsilon}{2} + \mathsf{K}\beta -\epsilon \leq  \frac{\epsilon}{2} + \frac{\epsilon}{2} -\epsilon=0,
		\end{aligned}
	\end{equation*} 
-If $i\in I_{\epsilon}(t,x)$, then by the descent lemma
	\begin{equation*}
		\begin{aligned}
			g_i(s,x+\beta v) &= g_i(s,x+\beta v) - g_i(t,x+\beta v) + g_i(t,x+\beta v)\\
			&\leq  -\frac{\kappa\beta}{2} + g_i(t,x) + \beta\langle \nabla g_i(t,x),v\rangle + \frac{L\beta^2}{2}\|v\|^2\\
			&\leq  -\frac{\kappa\beta}{2} + \beta\kappa -\frac{\kappa\beta}{2} = 0,
		\end{aligned}
	\end{equation*}
where we used $g_i(t,x)\leq 0$, $\langle \nabla g_i(t,x),v\rangle \leq \kappa$, and $L\beta^2\leq -\kappa \beta$.\\
Thus $x+\beta v\in C(s)$, hence $x\in C(s)+\beta \mathbb{B}$. We conclude  $C(t)\subset C(s) + \beta\mathbb{B}$. Exchanging the roles of $t$ and $s$ yields $C(s)\subset C(t)+\beta\mathbb{B}$. Therefore, $t\tto C(t)$ is Hausdorff continuous.\\
\emph{{\bf Claim 5.} For all $t\in [0,T]$, $C(t)$ is $\min\{\eta,-\kappa/L\}$-uniformly prox-regular.}\\
\emph{Proof of Claim 5.} Using $(b)$ and Claim 2, this follows from  \cite[Corollary 15.101]{Thibault-2023-II}.\qed
\endgroup





\end{document}